\documentclass[12pt]{amsart}
\usepackage{amssymb}
\usepackage{amsfonts}
\usepackage{eurosym}
\usepackage{amsmath}
\usepackage{amsfonts,xcolor}
\usepackage{amssymb}
\usepackage[all]{xy}
\usepackage{amssymb}
\usepackage{subfigure}
\usepackage{graphicx}
\usepackage[open,openlevel=1]{bookmark}

\theoremstyle{plain}
\newtheorem{theorem}{Theorem}[section]
\newtheorem{lemma}[theorem]{Lemma}
\newtheorem{proposition}[theorem]{Proposition}
\theoremstyle{definition}

\newtheorem{remark}[theorem]{Remark}
\newtheorem*{acknowledgement}{Acknowledgment}

\numberwithin{equation}{section}
\makeatletter
\def\filename{\texttt{\jobname.tex}}
\makeatother
\def\zbar{\overline{z}}
\def\wbar{\overline{w}}

\def\bF{\mathbf{F}} 
\def\fe{\mathfrak{e}} 
 
\def\cF{\mathcal{F}} 
\def\cL{\mathcal{L}} 
\def\cS{\mathcal{S}} 
\def\cW{\mathcal{W}} 
\def\cY{\mathcal{Y}} 
\def\cZ{\mathcal{Z}} 
\def\num{\mathcal{N}} 
\def\R{\mathbb{R}}
\def\C{\mathbb{C}}
\def\N{\mathbb{N}}

\def\E{\mathbb{E}}
\def\P{\mathbb{P}}
\def\1{\mathbf{1}}
\def\one{\mathbf{1}}
\def\disp{\displaystyle}
\def\spec{\mathrm{Spec}} 
\def\inlaw{\stackrel{d}{=}}
\def\inlawto{\stackrel{d}{\Longrightarrow}}
\def\Tr{\operatorname{Tr}}
\def\Var{\operatorname{Var}}
\def\Cov{\operatorname{Cov}}
\def\diag{\operatorname{diag}}

\def\nablaz{\nabla_z} 
 
\def\nablazup{\nabla_z^{\uparrow}} 
\def\nablawup{\nabla_w^{\uparrow}} 
\def\nablawbarup{\nabla_{\wbar}^{\uparrow}} 
\def\nablazdown{\nabla_z^{\downarrow}} 
\def\nablawdown{\nabla_w^{\downarrow}} 

 \newcommand{\cG}{\mathcal G}

\begin{document}
\title[Zero correlations and averaged fields of orthonormal
Gaussian Functions]{Zero correlations and averaged fields of
orthonormal Gaussian Functions}
\author{Lu\'{\i}s Daniel Abreu and Tomoyuki Shirai}
\address{NuHAG, Faculty of Mathematics, University of Vienna,
Oskar-Morgenstern-Platz 1, A-1090, Vienna, Austria\\
}
\address{Institute of Mathematics for Industry, Kyushu University, 744
Motooka, Nishi-ku, Fukuoka, 819-0395, Japan\\
}
\email{abreuluisdaniel@gmail.com}
\email{shirai@imi.kyushu-u.ac.jp}
\subjclass{Primary 60G55; Secondary 60G15, 30B20, 30H20, 42C40}
\keywords{Gaussian Entire Functions, Kac-Rice formulas, pair-correlations, Landau
Levels, polyanalytic functions, white noise spectrograms}


\begin{abstract}
We consider the family of point processes $\{\cZ_{f_{n}}\}_{n=0}^{\infty }$ of
zeros of Gaussian random functions $\{f_{n}(z,\zbar)\}_{n=0}^{\infty
} $, arising from the Gaussian Entire Function 
\begin{equation*}
f_{0}(z):=\sum_{k=0}^{\infty }\zeta _{k}\frac{z^{k}}{\sqrt{k!}}\text{, \ \ }
\zeta _{k}\sim N_{\C}(0,1)\text{ i.i.d.}
\end{equation*} 
by iteration of the Landau raising operator, and orthonormal
 at each point in expectation in the sense that
\begin{equation*}
\E\left[ e^{-\left\vert z\right\vert ^{2}}f_{n}(z,\bar{z})\overline{
f_{n^{\prime }}(z,\bar{z})}\right] ={\delta }_{nn'}. 
\end{equation*}
We first show that the normalized pair correlations $g_{n,n+k}(z,w)$ of the
 pairs $(\cZ_{f_{n}},\cZ_{f_{n+k}})$ exhibit \emph{a
 pattern reminiscent of the classical interlacing of zeros of
 orthogonal polynomials}: when $w\rightarrow z$, $g_{n,n+k}$ 
displays repulsion for $k=1$, attraction for $k=2$, 
and no short-range second-order correlation for $k \ge 3$. 
We complement this with the convergence of real-valued
 averaged fields on compacts $K \subset \C$,  
\[
\lim_{N\rightarrow \infty
 }\frac{1}{N}\sum_{n=0}^{N-1}\left\vert 
f_{n}(z,\zbar)e^{-\frac{\left\vert z\right\vert ^{2}}{2}}\right\vert
^{2}\rightarrow 1 \quad \text{ almost surely in $C(K)$}, 
\]
and a functional central limit theorem for the corresponding
 scaled fluctuations, which converge to the Gaussian
 process 
\[
 \cG(z) = \frac{1}{\sqrt{\pi}} \int_{\C} \one_{B(z,1)}(u)
 dW_{\R}(u), 
\]
 where $W_{\R}$ denotes real white noise on $\C$ and $B(z,1)$
 is the unit disk centered at $z$. 
The results are motivated by problems in signal processing. 
Due to an identification with white noise spectrograms, 
they confirm conjectures of Flandrin and
 Bayram-Baraniuk and provide a rationale for the efficiency
 of high resolution time-frequency algorithms, namely \emph{ConceFT}, by
 Daubechies, Wang and Wu. 
\end{abstract}

\maketitle

\section{Introduction}

The \emph{Gaussian Entire Function} (GEF) \cite{NS,GAFbook} is defined as 
\begin{equation}
f_{0}(z):=\sum_{k=0}^{\infty }\zeta_{k}\frac{z^{k}}{\sqrt{k!}}\text{, \ \ }
\zeta_{k}\sim N_{\C}(0,1)\text{ i.i.d. } z\in \C. 
\label{GEF}
\end{equation}
Despite its apparent simplicity, $f_{0}(z)$ enjoys interesting non-trivial
properties. The papers \cite{Sodin,NS,NSTransport} and the
monograph \cite{GAFbook} triggered different lines of research, raising a
considerable interest in complex analysis, probability and time-frequency
analysis. 
It has been proved in \cite{GhoshPeres} that the GEF is determined from its zeros
up to a random phase. 
In the applied side, the zero distribution has been used in the
filtering of time-frequency representations of signals embedded in white
noise
\cite{PNAS,Silence0,Maxima,levelSets,BFC,BH,Escudero,Flbook}. 

In this paper, we consider the infinite sequence $\{f_{n}(z,\zbar
)\}_{n=0}^{\infty }$ of Gaussian random functions
obtained from (\ref{GEF}) by repeated application 
of the Landau-level raising operator (we provide precise
definitions in the next subsection): 
\begin{equation}
f_{n}(z,\zbar):=\frac{1}{\sqrt{n!}}(\partial _{z}-\zbar
)^{n}f_{0}(z)=\sum_{k=0}^{\infty }\zeta _{k}\frac{(\partial _{z}-\zbar
)^{n}z^{k}}{\sqrt{n!k!}},  \quad z \in \C, 
\label{f_n}
\end{equation}
where $\zeta _{k}\sim N_{\C}(0,1)$ are i.i.d. 
For $n>0$, the functions $f_n$ are no longer analytic.  
Our results suggest a parallel between the distribution of zeros
of $f_{n}(z,\zbar)$ and the classical interlacing property for zeros of sequences
of orthogonal polynomials on a (possibly unbounded) interval $I\subseteq 
\R$, which goes back 150 years to the work of Markov and
Stieltjes. 

The classical interlacing property states that the zeros of orthogonal polynomials of
order $n$ and $n+1$ separate each other, becoming interlaced in this sense:
between two consecutive zeros of $p_{n+1}$ there is a zero
of $p_{n}$ 
(\cite[1.2-5]{Simon}, \cite[Theorem 2.2.3]{Ismail}). The existence of such a zero
suggests a \emph{negative pair correlation} between the zeros of $p_{n+1}$ and $
p_{n}$. 
By applying the property twice to $p_{n+2}$ it is easy to see that
it implies that there will be a zero of $p_{n}$ and a zero
of $p_{n+2}$ between two zeros of $p_{n+1}$, and this suggests\ a \emph{positive
correlation} between the zeros of $p_{n}$ and $p_{n+2}$. 
This motivates a weaker notion adapted to planar zero sets. 
We say that \emph{a sequence of functions}
$\{f_{n}\}_{n=0}^{\infty }$ satisfies a \emph{weak interlacing property}, if, 
for every integer $n$, there are \emph{negative pair correlations} in the short-range
limit between the
zeros of $f_{n+1}$ and $f_{n}$ and \emph{positive pair correlations} in the short-range
limit between the
zeros of $f_{n}$ and $f_{n+2}$. (This notion will be made
precise in Section~\ref{sec:main} in terms of normalized pair correlations.)
This definition applies to functions defined
in a planar domain, while the `strong' interlacing property of orthogonal
polynomials in $I\subseteq \R$ obviously does not, and we will show
that $\{f_{n}(z,\zbar)\}_{n=0}^{\infty }$ enjoys such
a weak interlacing property. 
This weak interlacing property suggests that the energy of
successive functions is placed in complementary regions and
that the averages of the squared intensities over many
consecutive functions fill the whole plane. We will show that this is indeed the case and
provide a refined analysis of this phenomenon.

\subsection{Overview and context}

The operator acting on $f_{0}(z)$\ to obtain $f_{n}(z,\zbar)$ is the
iterated inter Landau-level \emph{raising} operator
$\nablazup = \partial _{z}-\zbar$, 
normalized by the multiplicative factor
$(n!)^{-\frac{1}{2}}$. The Landau operator with a constant
perpendicular 
magnetic field, acting on the Hilbert space $L^{2}(\C,e^{-\left\vert
z\right\vert ^{2}})$, can be written as 
\begin{equation*}
\cL_{z,\zbar}:=-\partial _{z}\partial _{\zbar}+
\zbar\partial _{\zbar}=- 
\nablazup \nablazdown,  
\end{equation*}
where $\nablazup = \partial _{z}-\zbar\ $and
$\nablazdown =\partial _{\zbar}$. The purely discrete spectrum
of $\cL_{z,\zbar}$ (the Landau Levels) is 
$\sigma (\cL_{z,\zbar})=\{n:n=0,1,2,\ldots \}$. The
operators $\nablazup$\ and $\nablazdown$ are the \emph{inter
Landau levels raising and lowering operators}: 
$\nablazup$ (resp. $\nablazdown$) 
maps eigenfunctions with eigenvalue $n$ into eigenfunctions with eigenvalue $
n+1$ (resp. $n-1$).

The eigenspace associated with the eigenvalue $n$ arises from
the action of $\nablazup$ on
$\cF^{(0)}\left( \C\right) $ (the classical
Bargmann-Fock space of entire functions in
$L^{2}(\C,e^{-\left\vert z\right\vert ^{2}})$) (\cite{Bargmann}):  
\begin{equation*}
\cF^{(n)}\left( \C\right) = 
\frac{1}{\sqrt{n}}
\nablazup\cF^{(n-1)}\left(\C\right) \quad (n
\ge 1).
\end{equation*}
This defines an infinite sequence of subspaces of
$L^{2}(\C, e^{-\left\vert z\right\vert^{2}})$, $\left\{ \cF^{(n)}\left( 
\C\right) \right\}_{n=0}^{\infty}$ such that
\cite{VasiFock} 
\begin{equation*}
L^{2}(\C,e^{-\left\vert z\right\vert
^{2}})=\bigoplus\limits_{n=0}^{\infty
}\cF^{(n)}\left( \C \right). 
\end{equation*}
Each space $\cF^{(n)}\left( \C\right) $ possesses
a reproducing kernel given as \cite{AAZ,SHIRAI} 
\begin{equation}
K_{n}(z,w)= \frac{1}{n!} (\nablazup)^{n} 
(\nablawbarup)^{n} e^{z \wbar} 
= L_{n}(|z-w|^{2})e^{z\wbar},  
\label{KernelLandau}
\end{equation}
where $L_{n}(x)=L_{n}^{0}(x)$, with $L_{n}^{\alpha }$ denoting the Laguerre
polynomial of degree $n.$ See also \cite{AAZ,SHIRAI,ProgressGinibre}.
The Landau Levels provide a model for the perfect integer quantizations
observed in the integer quantum Hall effect
\cite{Nobel,Girvin}. As in \cite[Section 2]{ForrHon1999},
where $f_{0}(z)$ and the lowest Landau Level are 
discussed, $f_{n}(z,\zbar)$ can be seen as a random superposition of
states in the $nth$ Landau level. We also observe that 
the operators $\nablazup$\ and $\nablazdown$ coincide with the
symmetric and non-symmetric part of the Chern connection on $\C$ 
\cite{Feng}, which is given by $\nablaz = \nablazup + \nablazdown$. 
Finally, we point out that, 
although $f_n$ satisfies the Landau eigenvalue equation
$\cL_{z,\zbar}f_{n}=nf_{n}$ pointwise, 
it is almost surely not an $L^2$-eigenfunction of
$\cL_{z,\zbar}$ 
since $f_{n}\notin \cF^{(n)}\left( \C\right)
$ with probability one.

We will use the notation
\begin{equation*}
\cZ_{f}
=\left\{ z \in \C:f(z,\zbar)=0 \right\}. 
\end{equation*}
For the specific Gaussian functions considered here, we will
show in Section~\ref{subsec:isolated} that these zeros are almost surely
simple and isolated; hence $\cZ_{f}$ defines a point
process. As shown in \cite[Theorem 1.8]{GWHF},
using an approach based on time-frequency methods, the $1$-point intensity
of $\cZ_{f_{n}}$ with respect to Lebesgue measure is 
\begin{equation}
\rho _{\cZ_{f_{n}}}^{(1)}(z)=\frac{1}{\pi
 }(n+1/2+\frac{1}{4n+2}). 
\label{onepoint}
\end{equation}

Our main results provide the correlations between
$\cZ_{f_{n}}$ and 
$\cZ_{f_{n+k}}$, for all $n\geq 0$ and $k\geq 1$, 
in the short-range limit ($r \rightarrow 0$) and in the long-range limit ($r \rightarrow \infty $)
, where $r$ measures the distance of
pairs of points $(z,w)\in (\cZ_{f_{n}},
\cZ_{f_{n+k}})$. 
Our main results suggest a random planar analogue
of the interlacing property of zeros in the following sense: the pair
correlation of $(\cZ_{f_{n}},\cZ_{f_{n+1}})$
displays \emph{repulsion at small scales} (and no
interaction at large scales), 
so that zeros of consecutive elements of the sequence tend to separate each other;
the pair correlation of $(\cZ_{f_{n}},\cZ_{f_{n+2}})$
displays \emph{attraction at small scales,} but, for $k>2$, $(\cZ_{f_{n}},\cZ_{f_{n+k}})$ are \emph{uncorrelated at
small scales}. 
As far as we know, these are the first results providing information
about the interactions among all possible pairs of different elements of an
infinite family of orthogonal Gaussian processes.

The notation $f_{n}(z,\zbar)$ is used to emphasize
the dependence 
in $z$ and $\zbar$, since for $n>0$, it is not an entire function, it
only satisfies the weaker polyanalytic condition $\partial _{\zbar
}^{n+1}f_{n}(z,\zbar)=0$ \cite{Abr2010,Balk,HendHaimi,VasiFock}.
Equivalently, they are polynomials $F(z,\zbar)$ of degree $n$ in $
\zbar$, whose coefficients $F_{i}(z)$ are analytic functions:
\begin{equation}
F(z,\zbar)=F_{0}(z)+\zbar F_{1}(z)+...+\zbar^{n}F_{n}(z). 
\label{poly}
\end{equation}

The zeros of these functions are far from being understood
(see, e.g., \cite{PolyanalyticZeros} for a recent study on
the topic). In general, $F(z,\zbar)$ can have
both isolated points and non-discrete curves as zeros 
(e.g. the zero set of $F(z,\zbar)=z-\zbar z^2$ of degree $1$
is $\{0\} \cup \{|z|=1\}$). The
processes considered in this paper 
can be seen as a `Gaussian functions analogue' of the determinantal point
processes in higher Landau Levels
\cite{HaiAron,abgrro17,HendHaimi,HaiHen2,SHIRAI,MakotoShirai,Deviations},
usually referred to as polyanalytic ensembles.

To understand the interactions between the zeros of \emph{pairs of different
Gaussian processes}
$(\cZ_{f_{n}},\cZ_{f_{n+k}})$, for all
$n\geq 0$ and $k\geq 1$ we start with adapted Kac-Rice
formulas \cite{level}. 
Kac-Rice formulas were used to study averages of critical points of
non-holomorphic Gaussian processes in compact Riemann surfaces arising in
String Theory \cite{DZS}, and the associated value
distribution \cite{FZ,FZ1}. This was adapted to confirm a
prediction of Flandrin \cite{Fl1} about local maxima of
white noise spectrograms and their value distribution
\cite{Maxima}. 
A well-known application of the Kac-Rice formula is the
analysis of critical points of smooth Gaussian functions on
the $N$-dimensional sphere \cite{AB13}. 

To compute the correlations between $\cZ_{f}$ and $\cZ_{g}$,
with $f\neq g$, we will derive the following Kac-Rice formula, which is
adapted from \cite[Exercise 6.1]{level} to our setting: 
\begin{equation}
\rho _{\cZ_{f},\cZ_{g}}^{(2)}(z,w)
=\E\left[
\left\vert |\nablazup f|^{2}-|\nablazdown f|^{2} 
\right\vert \cdot \left\vert |\nablawup g|^{2}-|\nablawdown g|^{2}\right\vert \text{ }|\ f(z)=g(w)=0\right] \cdot
p_{(f(z),g(w))}(0,0). 
\label{Kak-Rice}
\end{equation}
To compute this conditional expectation, 
we introduce a $6$-dimensional complex Gaussian
vector and its covariance matrix.  
While this leads to intractable formulas for $\rho _{\cZ_{f_{n}},
\cZ_{f_{n+k}}}^{(2)}(z,w)$, we found a way of using the structure of
the Landau equation to compute the short-range limits $\lim_{z\rightarrow w}\rho _{
\cZ_{f_{n}},\cZ_{f_{n+k}}}^{(2)}(z,w)$ and
the long-range limit $\lim_{|z-w| \to \infty }\rho_{\cZ_{f_{n}},\cZ
_{f_{n+k}}}^{(2)}(z,w)$ without resorting to symbolic algebra.

\section{Main results}\label{sec:main}

In this section, we present our findings on zero
correlations and behavior of averaged fields. 
Two central results are Theorem~\ref{Main}, which
describes the correlation patterns of zeros, and 
Theorem~\ref{thm:functional-clt}, which establishes a
functional central limit theorem for the 
scaled fluctuations of the averages, showing that they 
converge to an interesting Gaussian process in the limit. 

\subsection{Correlations between $\cZ_{f_{n}}$ and $\cZ_{f_{n+k}}$}

Denote by $\rho_{n,n+k}^{(2)}(z,w)$ the $2$-point intensity
of the pair of point processes 
$(\cZ_{f_{n}}, \cZ_{f_{n+k}})$ with respect
to Lebesgue measure: 
\begin{equation*}
 \E \left[ \sum_{a \in \cZ_{f_{n}}}
 \delta_a(dz) \otimes \sum_{b \in \cZ_{f_{n+k}}} \delta_b(dw)
    \right]
= \rho_{n,n+k}^{(2)}(z,w)dzdw. 
\end{equation*}
We define the \emph{normalized} \emph{pair correlation function }$
g_{n,n+k}(z,w)$ by\emph{\ } 
\begin{equation*}
g_{n,n+k}(z,w):=\frac{\rho _{n,n+k}^{(2)}(z,w)}{\rho _{n}^{(1)}(z)\rho
_{n+k}^{(1)}(w)}. 
\end{equation*}
The normalized pair correlation function provides the following information.

\begin{enumerate}
\item If $g_{n,n+k}(z,w)<1$, $\cZ_{f_{n}}$ and
      $\cZ_{f_{n+k}} $ have \emph{negative pair 
      correlations} (points in $\cZ_{f_{n}}$ and
      $\cZ_{f_{n+k}}$ tend to \emph{repel} each
      other at $(z,w)$). 

\item If $g_{n,n+k}(z,w)=1$, $\cZ_{f_{n}}$ and
      $\cZ_{f_{n+k}}$ are \emph{uncorrelated} at
      $(z,w)$ at the level of the pair correlation. 

\item If $g_{n,n+k}(z,w)>1$, $\cZ_{f_{n}}$ and $\cZ
_{f_{n+k}} $ have \emph{positive pair correlations} (points in
      $\cZ_{f_{n}}$ and $\cZ_{f_{n+k}}$ tend
      to \emph{attract} each other at $(z,w)$). 
\end{enumerate}

From this point on, we will make free use of the following notational
convention for the short-range limit of the correlations
between $\cZ_{f_{n}}$ and $\cZ_{f_{n+k}}$:
\begin{equation}
g_{n,n+k}(z,z) := \lim_{w\rightarrow z,w\not=z}g_{n,n+k}(z,w), 
\label{eq:gzz}
\end{equation}
that is, $g_{n,n+k}(z,z)$ is the off-diagonal contact
limit. 
Our main result provides the explicit values of $g_{n,n+k}(z,z)$ for all $n$
and $k$ (see \eqref{precise}
in Section~\ref{subsec:proofoftheorem1} for the precise
values of $g_{n,n+k}(z,z)$). 
To facilitate the recognition of the correlation pattern, we
only provide the asymptotic expansion for large $n$. 
The exact formulas in \eqref{precise} imply the stated inequalities
for all $n\geq 0$; the theorem displays their large-$n$
expansions. 

\begin{theorem}\label{Main} 
For fixed $k \ge 1$, as $n\rightarrow \infty $, 
\begin{equation*}
g_{n,n+k}(z,z) = 
\begin{cases}
\disp 1-\frac{5}{4n^{2}}+O\left( \frac{1}{n^{3}}\right) <1 &
 \text{if $k=1$}, \\
\disp \frac{4}{3}
+\frac{1}{12n^2}+O\left( \frac{1}{n^3}\right) 
>1 &
 \text{if $k=2$},  \\
\disp 1 & \text{if $k\geq 3$}.  
\end{cases}
\end{equation*}
Thus, the short-range correlations between $\cZ_{f_{n}}$ and $
\cZ_{f_{n+k}}$ display repulsion if $k=1$, attraction if $k=2$ and
no correlations if $k\geq 3$. 
\end{theorem}

\begin{remark}\label{rem:corr=1}
For every $n\ge 0$ and every $k\ge 3$,
\[
g_{n,n+k}(z,z)=1.
\]
Thus the absence of contact interaction for gaps $k\ge 3$ is exact, and not
merely asymptotic.
\end{remark}

We observe from (\ref{precise}) that there is
\emph{outstanding repulsion} 
$g_{n,n+1}(z,z)=0$ if and only if $n=0$. 
Moreover, $g_{n,n+1}(z,z)$ increases with $n$ and tends
to $1$. For $k=2$, one has $g_{n,n+2}(z,z)>1$ for every
$n\geq 0$. The first values are $g_{0,2}=79/65$,
$g_{1,3}=5821/4375$, and $g_{2,4}=9041/6765$; the maximum is
attained at $n=2$, and $g_{n,n+2}(z,z)$ is strictly
decreasing for $n\geq 2$, with limit $4/3$. 
Curiously, $1$ and $4/3$ are the expected number of zeros and saddle
points of $f_0(z)$, respectively \cite{Maxima}. Although we do not
have a conceptual explanation for this coincidence, the appearance
of $4/3$ may have a structural origin. See
Remark~\ref{rem:4/3}. 

Theorem~\ref{Main} shows that the zeros of consecutive 
functions (and consequently, their energy patches) tend to be
placed in complementary regions of the plane, suggesting
that, for large $N$, their average fills the whole plane. 
This suggests that the limit of 
\begin{equation}
S_N(z)=\frac{1}{N}\sum_{n=0}^{N-1}\left\vert e^{-\frac{\left\vert
z\right\vert ^{2}}{2}}f_{n}(z,\zbar)\right\vert ^{2},
\label{average0}
\end{equation}
should approach $1$. This has been conjectured in
\cite[Section 10.2]{Flbook}, using similar considerations
with local maxima instead of zeros (see
Section~\ref{subsec:time-frequency} for the time-frequency
interpretation of the results.) 
Theorem~\ref{Main} will be complemented by a proof of this
conjecture: we will show that 
$S_N(z)\rightarrow 1$, 
first pointwise almost surely, and then, on compact sets, 
almost surely as a continuous process on $\C$. 
The precise statements are presented in the next subsection. 

\begin{figure}[tbp]
\centering
\includegraphics[scale=0.35]{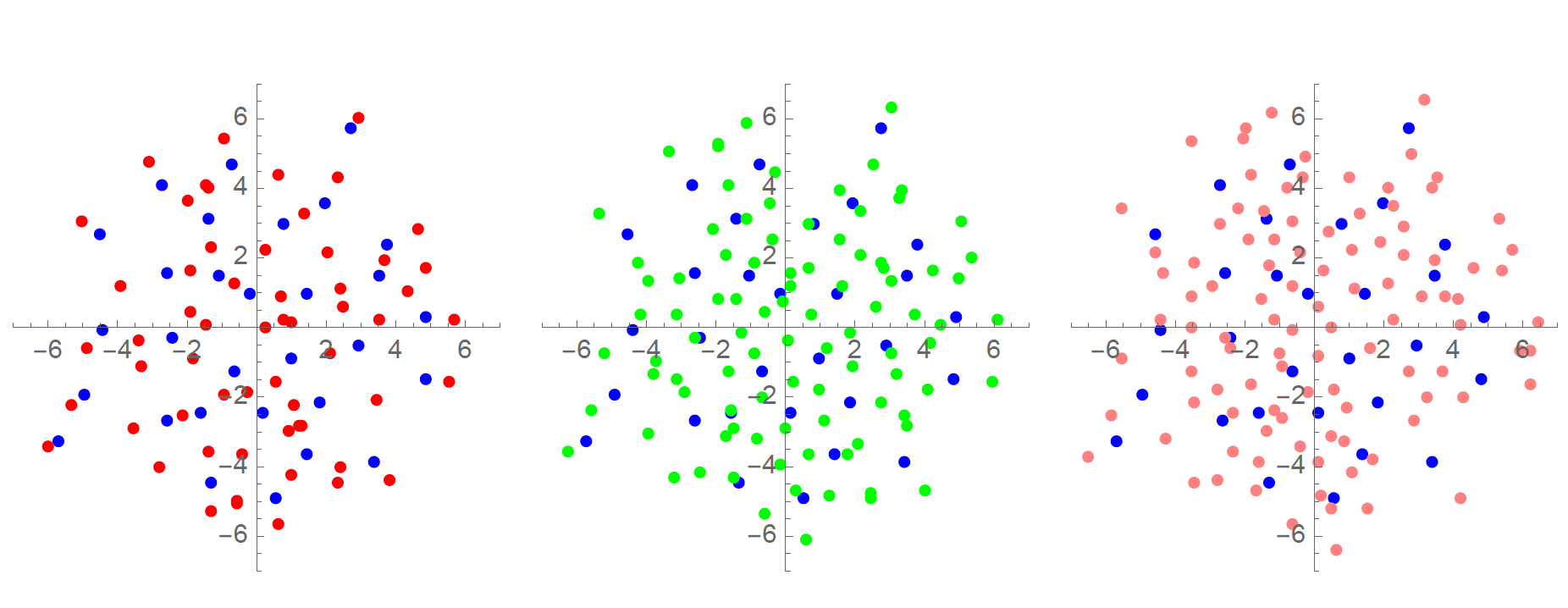}
\caption{The zeros of $f_k^{(30)}(z,\zbar)$ for $k=0$ (blue), $k=1$
(red), $k=2$ (green), and $k=3$ (pink), where $f_k^{(N)}(z,\zbar)$ is
the truncated polynomial of $f_k(z,\zbar)$ up to degree $N$. The
numbers of zeros are $30$ ($k=0$), $59$ ($k=1$), $94$ ($k=2$), $103$ ($k=3$). From left to
right, $(0,1)$, $(0,2)$ and $(0,3)$ correlations are illustrated.}
\label{fig:1}
\end{figure}

Theorem~\ref{Main} is illustrated in Figure~\ref{fig:1} for the cases $n=0$ and $k=1,2,3$,
where zeros of pairs of consecutive truncated random polynomials obtained
from $f_{n}(z,\zbar)$ are superimposed. The first picture on the left
illustrates the outstanding repulsion between
$\cZ_{f_{0}}$ and $\cZ_{f_{1}}$. The second
picture shows the attraction between $\cZ_{f_{0}}$
and $\cZ_{f_{2}}$, while the neutral 
correlations between $\cZ_{f_{0}}$ and $\cZ_{f_{3}}$ are
visible in the third picture. The apparent clustering in the
third picture is consistent with 
the absence of short-range second-order correlation,
together with the increase in the number of
zeros and the weaker self-repulsion 
within $\cZ_{f_{n}}$ as $n$
increases (this was shown for $n=1$ in \cite{Baber}).

The next result gives the long-range limit of the
correlations between $\cZ_{f_{n}}$ and
$\cZ_{f_{n+k}}$. As expected, for every $n \ge 0$ 
and $k \ge 1$, 
long-range correlations between zeros vanish in the limit.

\begin{theorem}\label{LongRange} 
As $|z-w| \rightarrow \infty $, for every
 $n \ge 0$ and $k \ge 1$, 
\begin{equation*}
g_{n,n+k}(z,w)\rightarrow 1. 
\end{equation*}
\end{theorem}

\subsection{Limits of averages}

We first observe that, due to
\eqref{eq:discrete-orthogonality}, 
\begin{equation*}
\E\left[ \frac{1}{N}\sum_{n=0}^{N-1}\left\vert e^{-\frac{\left\vert
z\right\vert ^{2}}{2}}f_{n}(z,\zbar)\right\vert
^{2}\right] =1. 
\end{equation*}
By orthogonality and joint Gaussianity, for each fixed $z \in \C$,  
$\{|e^{-{|z|^2/2}} f_n(z,\zbar)|^2\}_{n \ge 0}$ are
i.i.d. exponential random variables with mean $1$. 
Hence, it immediately follows that the law of large numbers 
\begin{equation}
\lim_{N\rightarrow \infty} 
\frac{1}{N}\sum_{n=0}^{N-1}
\left\vert f_{n}(z,\zbar)e^{-\frac{\left\vert z \right\vert^{2}}{2}}
\right\vert^{2}\rightarrow 1\text{ a.s. }, 
\label{eq:1} 
\end{equation}
and the central limit theorem 
\begin{equation}
\frac{\sum_{n=0}^{N-1}\left\vert f_{n}(z,\zbar)e^{-\frac{\left\vert
z\right\vert^{2}}{2}}\right\vert^{2}-N}{\sqrt{N}} \inlawto N(0,1). 
\label{eq:2} 
\end{equation}
hold. We now strengthen \eqref{eq:1} and \eqref{eq:2} to the process level, 
and establish the functional law of large numbers, and after
appropriate rescaling, the functional central limit theorem.

In what follows, we use the following notation. 
\[
X_{n}(z):=e^{-|z|^{2}/2}\,f_{n}(z,\bar{z}),\qquad S_{N}(z):=\frac{1}{N}
\sum_{n=0}^{N-1}|X_{n}(z)|^{2}\quad (z\in \C). 
\]

\begin{theorem}\label{thm:functionalLLN}
For any compact $K\subset \C$, we have 
\begin{equation*}
S_{N} \to 
1\quad \text{a.s. in }C(K),\quad \text{i.e.}\quad
\sup_{z\in K}|S_{N}(z)-1| 
\to 
0 \quad \text{a.s.} \qquad (N\rightarrow
\infty). 
\end{equation*}
Here $C(K)$ is the space of continuous functions on $K$. 
\end{theorem}
In fact, a stronger H\"older-space version is proved in
Section~\ref{sec:proof34}; see Theorem~\ref{thm:holder}. 

For $z_0 \in \C$ we define the rescaled field $\cG_N^{(z_0)}
= \{\cG_N^{(z_0)}(\xi)\}_{\xi \in \C}$ by 
\[
\cG_N^{(z_0)}(\xi) := 
\sqrt{N} (S_N(z_0+\sqrt{N}\,\xi) -1),
\qquad
\xi \in \C.
\]
We will prove a functional central limit theorem for the
real-valued field $\cG_N^{(z_0)}$. 

\begin{theorem}\label{thm:functional-clt}
Fix $z_0 \in \C$ and a compact set $K \subset \C$. Then
\[
\cG_N^{(z_0)} \inlawto \cG
\qquad
\text{in } C(K),
\]
where $\cG$ is a centered real Gaussian process on $\C$ with covariance
\[
\E\bigl[\cG(\xi)\cG(\eta)\bigr] = \kappa(|\xi-\eta|).
\]
The limit covariance is the normalized overlap area of two unit
disks,
\[
\kappa(r) := \frac{|B(0,1)\cap B(r,1)|}{\pi},
\qquad
r\ge 0,
\]
where $B(z, \rho)$ is the disk of radius $\rho \ge 0$
centered at $z \in \C$. Equivalently,
\[
\kappa(r)=
\begin{cases}
\disp \frac{2}{\pi}\arccos \frac{r}{2}
-
\frac{r}{\pi}\sqrt{1-\frac{r^2}{4}},
&
0\le r\le 2,
\\[1ex]
0,
&
r\ge 2.
\end{cases}
\]
Moreover, for every $0<\alpha<\frac{1}{2}$, the family
$\{\cG_N^{(z_0)}\}_{N\ge 1}$ is tight in $C^{0,\alpha}(K)$, and $\cG$ admits a
version with sample paths in $C^{0,\alpha}(K)$.
\end{theorem}

\begin{remark}
The limiting Gaussian process admits the Wiener-integral representation 
\[
 \cG(z) = \frac{1}{\sqrt{\pi}} \int_{\C} \1_{B(z,1)}(u) dW_{\R}(u). 
\]
where $W_{\R}$ denotes real white noise on $\C$. 
This covariance kernel also appears as the high-intensity
 Gaussian limit of local counting fluctuation fields for a planar
 Poisson point process. Indeed, if $\eta_\lambda$ is a
 homogeneous Poisson point process on $\R^2$ with intensity
 $\lambda/\pi$, and 
\[
Y_\lambda(z):=\frac{\eta_\lambda(B(z,1))- \lambda}{\sqrt{\lambda}},
\]
then for any $z,w\in \R^2$,
\[
\Cov\bigl(Y_\lambda(z),Y_\lambda(w)\bigr)
= \frac{|B(z,1)\cap B(w,1)|}{\pi} .
\]
Moreover, as $\lambda\to\infty$, the finite-dimensional
 distributions of $\{Y_\lambda(z)\}_{z\in\R^2}$ converge to
 those of a centered Gaussian field with this
 covariance. Thus, the kernel $\kappa(z,w)=|B(z,1)\cap B(w,1)|/\pi$
 describes the limiting overlap of local counting noise seen
 through unit disk windows. 
\end{remark}

\subsection{Remarks}
A striking feature of Theorem~\ref{Main} is 
jump from $g_{0,1}(z,z)=0 $ to $g_{1,2}(z,z) = 9/13$. The outstanding
repulsion in the $n=0$ case is due to the analyticity of $f_{0}(z)$, which
ensures that $f_{0}(z)$ and $f_{1}(z,\zbar)$\
cannot have a common zero. 
Indeed, if $f_0(z_0) = f_1(z_0, \zbar_0) =0$, then
$f_1(z_0, \zbar_0) = \partial_z f_0(z_0) -
\overline{z_0} f_0(z_0) = \partial_z f_0(z_0) =0$, so 
$z_0$ is a multiple zero of $f_0(z)$. This is 
impossible almost surely since $f_0(z)$ is entire and has no multiple
zeros almost surely. 
This also admits a geometric interpretation: zeros of 
$f_0$ correspond to local minima of $\log\big| e^{-|z|^2/2}
f_0(z) \big|$, while zeros of $f_1$ correspond to its non-minimal
critical points, namely, saddle points and local maxima. 
This explains heuristically the exceptional
strength of the repulsion in the case $n=0$. For $n \ge 1$,
such an argument is no longer available since $f_n$ is not
analytic. 
These heuristics justify the strong repulsion between zeros of
$f_{0}(z)$ and $f_{1}(z,\zbar)$. However, this
property is exclusive of the case $n=0$, since for $n\geq
1$, the Gaussian functions are not analytic (the maximum 
principle for polyanalytic functions \cite[Lemma 2.1]{Balk,PlanarSampling}
does not seem to assure that local minima are zeros).

The problem studied in this paper admits several natural variations, 
which may be worth considering elsewhere. A question not covered by our results is the
study of the correlations between pairs of zeros $z,w\in
\cZ_{f_{n}}$ 
of $f_{n}(z,\zbar)$ (see \cite{Baber} for the $n=1$ case), which are
likely to increase with $n$, due to (\ref{onepoint}). Further questions are
suggested by the conditional expectations of a zero process, given a zero of
another process \cite{Feng2}.

\subsection{Outline of the paper}

The paper is organized as follows. 
Section~\ref{sec:motivation} discusses the motivations. 
Section~\ref{sec:property} collects the basic properties 
of the sequence $\left\{ e^{-\frac{\left\vert z\right\vert
^{2}}{2}}f_{n}(z,\zbar)\right\}_{n=0}^{\infty }$ 
and introduces the notion and preliminary results used
later. 
Sections~\ref{sec:proof12} and \ref{sec:proof34} contain 
the proofs of the main results. 
Section~\ref{sec:FCLT} contains the proof of
Theorem~\ref{thm:functional-clt}. 
Section~\ref{sec:whitenoise} interprets the results in
terms of zeros of white noise spectrograms, and shows that this supports the
rationale behind the use of spectrogram averages in high-resolution
time-frequency analysis.

\section{Motivations of the work}\label{sec:motivation}

We outline our motivations in this section.

\begin{enumerate}
\item \emph{Interactions between Landau levels}. The correlations within a
single layer of electrons formed in \emph{the integer
      quantum Hall effect} 
\cite{Nobel} can be modeled by the determinantal processes
      associated with the kernel of the Landau levels
      \cite{HaiAron,abgrro17,HendHaimi,HaiHen2,SHIRAI,MakotoShirai,Deviations}. 
Theorem \ref{Main} sheds light on interaction patterns between random
superpositions of states in \emph{different} Landau levels.

\item \emph{Interactions between zeros and critical points
      of GEF's} \cite{Hanin,Feng2,Feng}. We observe that the
      operators $\nablazup$ and $\nablazdown$ coincide with the symmetric and 
non-symmetric parts of the Chern connection on $\C$
      \cite{Feng},
      $\nablaz = \nablazup + \nablazdown$. 
      Feng \cite{Feng}
      investigated short/long-range correlations between
      zeros and critical points of the GEF under the Chern
      connection (which, when acting on entire functions
      boils down to $\nablazup$). Our results
      for $n=0$ and $k=1$ recover the $1$-dimensional
      version of the universal case of \cite{Feng}.

\item \emph{White noise spectrograms}. 
The sequence $\{f_{n}(z,\zbar) \}_{n=0}^{\infty }$ has the
      same zeros as spectrograms with Hermite windows (see
      Section~\ref{sec:whitenoise} for definition of
      $\spec_{h_{n}}$) up to the natural phase-space
      identification and the scaling $z\mapsto \bar z/\sqrt{\pi}$. In \cite[Section 10.2]{Flbook},
      empirical and heuristic reasoning suggest that the
      allocations of energy patches of white noise
      spectrograms with Hermite functions have an
      intertwining pattern (supported by our main results),
      which also support the numerical evidence of high
      resolution time-frequency methods  
      \cite{XFlandrin,ConceFT,ImpConcEFT}.  

\item \emph{Wigner-Ville \cite{FranzH} time-varying spectral estimation}. In 
\cite{BB}, Bayram and Baraniuk suggested estimating a time-varying random
signal $\nu $ using the average 
\begin{equation}
S_{N}(\nu )(z)=\frac{1}{N}\sum_{n=0}^{N-1}\spec_{h_{n}}\nu
 (z), 
\label{BB_0}
\end{equation}
but, to the best of our knowledge, no mathematical proof is
      available showing that $S_{N}(\nu )$ approaches the
time-varying spectrum of $\nu $. Using the connections
      described in Section~\ref{sec:whitenoise},
      (\ref{average0})\ can be written as
      $S_{N}(z)=S_{N}(\cW)(z)$, where 
$\cW$ is white noise. By Theorem~\ref{thm:functionalLLN}, at least in the white
noise case $\nu =\cW$, the conjectured convergence 
$S_{N}(\cW)(z)\rightarrow 1$ ($1$ is the time-varying spectrum of
      $\cW$) almost surely in $C(K)$ for every compact $K\subset \C$.
\end{enumerate}

\section{Basic properties}\label{sec:property}

In this section, we compute the 
covariance kernel for $\{f_{n}(z,\zbar)\}_{n = 0}^{\infty}$, 
expand $f_{n}(z,\zbar)$ in the natural basis of the
Landau Level eigenspace (the complex Hermite polynomials defined below), and
clarify in what sense the family $\left\{ e^{-\frac{\left\vert
z\right\vert^{2}}{2}}f_{n}(z,\zbar)\right\}
_{n=0}^{\infty }$ is a sequence of 
orthonormal Gaussian functions. In the final subsection, 
we prove directly that, for every $n \ge 0$, the zeros of
$f_{n}(z,\zbar)$ are almost surely simple and isolated. 

\subsection{Covariance kernel}

Since the correlations we are going to investigate depend on $K_{n}(z,w)$,
the covariance kernels of $f_{n}(z,\zbar)$, we start this section by
computing $K_{n}(z,w)$. The covariance kernel of $f_{0}(z)$,
viewed as a Gaussian process, is well known to be given by $K_{0}(z,w)=e^{z\wbar}$, the
reproducing kernel of the lowest Landau level eigenspace (the Fock space of
entire functions). This observation extends to higher Landau Level
eigenspaces, since by definition of covariance kernel, 
\begin{align*}
K_{n}(z,w)&= 
\E[f_{n}(z,\zbar)\overline{f_{n}(w,\wbar)}]
=\frac{1}{n!}\E[(\nablazup)^{n}f_{0}(z)
 \overline{(\nablawup)^{n}f_{0}(w)}].  
\end{align*}
Using the kernel of $f_0$, we obtain 
\[
K_{n}(z,w) = 
\frac{1}{n!}(\nablazup)^{n}(\nablawbarup)^{n}K_{0}(z,w) 
= L_{n}(|z-w|^{2})e^{z\wbar}.  
\]
Thus $f_n(z,\zbar)$ is a Gaussian process on $\C$
whose covariance kernel coincides with 
the reproducing kernel of the $n$th Landau level
eigenspace (\ref{KernelLandau}). Note that $f_{n}(z,\zbar)$ is smooth
but not analytic unless $n=0$. It is actually polyanalytic, 
satisfying $\partial_{\zbar}^{n+1}f_{n}(z,\zbar)=0$ (the Landau
levels eigenspaces are also known as true/pure polyanalytic
Fock spaces \cite{VasiFock, Abr2010,AAZ,HendHaimi}). 
The same kernel also appears as the
correlation kernel of the corresponding polyanalytic
determinantal point processes in higher Landau levels
\cite{HaiAron,abgrro17,HendHaimi,SHIRAI,MakotoShirai,Deviations}.  

\subsection{Expansion in complex Hermite polynomials}

Consider the following sequence of Gaussian processes, 
\begin{equation*}
f_{n}(z,\zbar):=\frac{1}{\sqrt{n!}}(\nablazup)^{n}f_{0}(z), 
\end{equation*}
where $f_{0}(z)$ is the Gaussian Entire Function defined in (\ref{GEF}). 
The complex Hermite polynomials are defined as \cite{Ghanmi,IsmailCH}, 
\begin{equation}
H_{k,n}(z,\zbar)=\frac{(\nablazup)^{n}z^{k}}{\sqrt{n!}
\sqrt{k!}}=\left\{ 
\begin{tabular}{l}
${\sqrt{\frac{n!}{k!}}z^{k-n}L_{n}^{k-n}\left( \left\vert z\right\vert
^{2}\right) ,\qquad k>n\geq 0,}$ \\ 
$\left( -1\right) ^{n-k}{\sqrt{\frac{k!}{n!}}\zbar^{n-k}L_{k}^{n-k}
\left( \left\vert z\right\vert ^{2}\right) ,\qquad 0\leq
 k\leq n}$, 
\end{tabular}
\right. , 
\label{ComplexHermite}
\end{equation}
where the Laguerre polynomial $L_{n}^{\alpha }(x)$ is 
\begin{equation*}
L_{n}^{\alpha}(x) = 
\sum_{j=0}^{n}(-1)^{j} 
\binom{n+\alpha}{n-j}\frac{x^j}{j!},\qquad x\in \R, \quad
 n\geq 0, n+\alpha \geq 0. 
\end{equation*}
Combining (\ref{GEF}), (\ref{f_n}) and (\ref{ComplexHermite}) leads to the
expansion 
\begin{equation}
f_{n}(z,\zbar):=\frac{1}{\sqrt{n!}}(\nablazup)^{n}f_{0}(z) =
 \sum_{k=0}^{\infty }\zeta
 _{k}\frac{(\nablazup)^{n}z^{k}}{\sqrt{n!}\sqrt{k!}} 
= \sum_{k=0}^{\infty }\zeta_{k}H_{k,n}(z,\zbar).  
\label{FnHermite}
\end{equation}

\subsection{Orthonormality}\label{subsec:orthonormal}

From the doubly indexed orthogonality relation  
\cite{Ghanmi,IsmailCH}
\begin{equation*}
\int_{\C}H_{k,n}(z,\zbar)\overline{H_{k', n'}(z,\zbar)}e^{-\left\vert z\right\vert
 ^{2}}\frac{dz}{\pi }={\delta }_{kk'}{\delta}_{nn'}, 
\end{equation*}
one obtains a formal orthogonality relation between different
Landau-level indices after integration against
$e^{-|z|^2}dz/\pi$:  
\begin{equation*}
\int_{\C}f_{n}(z,\zbar)\overline{f_{n'}(z,\zbar)}
e^{-\left\vert z\right\vert ^{2}}\frac{dz}{\pi }= \text{ `}{
\infty \cdot \delta }
_{nn'}\text{'}.  
\end{equation*}
 This cannot be interpreted as samplewise
 $L^2$-orthonormality, because the diagonal integral diverges
 almost surely. 
As in Remark~\ref{rem:unitary} below, it is easily seen that 
the complex Hermite polynomials also enjoy a discrete orthogonality
relation with respect to the Landau Level index:
\begin{equation*}
\sum_{k=0}^{\infty }H_{n,k}(z,\bar{z})\overline{H_{n',k}(z,\bar{z})} 
=e^{|z|^2} \delta_{nn'}.  
\end{equation*}
Using this discrete orthogonality, we see that 
\begin{equation}
\E\left[ e^{-\left\vert z\right\vert^{2}}f_{n}(z,\bar{z}) 
\overline{f_{n'}(z,\bar{z})}\right] = \delta_{nn'}. 
\label{eq:discrete-orthogonality}
\end{equation}
Thus, for fixed $z \in \C$, the family of Gaussian functions
$\left\{ e^{-\frac{|z|^{2}}{2}} 
f_{n}(z,\zbar)\right\}_{n=0}^{\infty}$ is
orthonormal in $L^2(\P)$. 
In Subsection~\ref{subsec:pointwise}, we will see the
covariance structure for $z,w \in \C$.  

\begin{lemma}\label{lem:displacement-matrix}
For $z \in \C$, define
\[
U(z)_{m,n}:=e^{-|z|^2/2}H_{m,n}(z,\zbar),
\qquad m,n\in\mathbb N_0.
\]
Then $U(z)=(U(z)_{m,n})_{m,n\ge 0}$ defines a unitary operator on
$\ell^2(\mathbb N_0)$.
\end{lemma}
\begin{proof}
Let $\{e_n\}_{n\ge0}$ be the canonical orthonormal basis of $\ell^2(\N_0)$, and
let $a$ and $a^*$ be the annihilation and creation operators,
\begin{equation}
a e_0=0,\qquad a e_n=\sqrt n \, e_{n-1}\quad (n\ge1),
\qquad
a^* e_n=\sqrt{n+1} \, e_{n+1}.
\label{eq:creation-annihilation} 
\end{equation}
For $z \in \C$, let
$D(z):=\exp \bigl(z a^*-\zbar a\bigr)$ 
be the Weyl displacement operator. It is standard that $D(z)$ is unitary and that,
in the Fock basis $\{e_n\}_{n \in \N_0}$, its matrix coefficients are given by
\begin{equation*}
\langle e_m, D(z)e_n\rangle
=
e^{-|z|^2/2}H_{m,n}(z,\zbar),
\qquad m,n \in \N_0;
\end{equation*}
see, for instance, Fujii \cite[Eq.~(27), (28)]{Fujii} and 
\eqref{ComplexHermite}. 
Hence $U(z)$ is precisely the matrix representation of $D(z)$, and is therefore unitary.
\end{proof}

\begin{remark}\label{rem:unitary}
Since $U(z)$ is unitary, its columns are orthonormal. Therefore
\[
\sum_{m\ge 0} U(z)_{m,n'}\overline{U(z)_{m,n}}=\delta_{nn'}, 
\]
which is equivalent to \eqref{eq:discrete-orthogonality}.  
\end{remark}

From Lemma~\ref{lem:displacement-matrix}, 
for $n \in \N_0$ and $z \in \C$, one can define an $\N_0$-valued
 random variable $M_{n,z}$ such that 
\begin{equation*}
\P(M_{n,z}=m):=|U(z)_{m,n}|^2 = |\langle e_m, D(z)e_n \rangle|^2,
\qquad m \in \N_0. 
\label{eq:defMnz} 
\end{equation*}

\begin{remark}\label{rem:spectral-measure-Mnz}
We recall the following interpretation of the law of $M_{n,z}$.
Let $\num:=a^*a$ be the number operator on $\ell^2(\N_0)$.
Since the number operator $\mathcal N$ is diagonal in the
basis $\{e_m\}_{m\ge 0}$ with $\mathcal N e_m=me_m$, its spectral projection
onto a Borel set $B\subset \mathbb R$ is simply
\[
E_{\mathcal N}(B)=\sum_{m\in B\cap \mathbb N_0} e_m \otimes
e_m,  
\]
where $e_m \otimes e_m$ is the orthogonal projection onto
the space $\C e_m$. 
Thus, if the system is in the state $\psi$, the spectral measure of
$\mathcal N$ associated with $\psi$ assigns to $B$ the
 probability 
\[
\mu_{\psi}(B) := \langle E_{\mathcal N}(B)\psi,\psi\rangle
=
\sum_{m\in B\cap \mathbb N_0}|\langle e_m,\psi\rangle|^2.
\]
In particular, for $\psi=D(z)e_n$ this gives
\[
\mu_{D(z)e_n}(B)
=
\sum_{m\in B\cap \mathbb N_0}
|\langle e_m,D(z)e_n\rangle|^2
=
\mathbb P(M_{n,z}\in B).
\]
Hence the law of $M_{n,z}$ is exactly the scalar spectral measure of
$\mathcal N$ associated with the vector $D(z)e_n$.
In probabilistic terms, the spectral projections of $\num$ play the role of
the events specifying the possible values of the random
 variable. 
\end{remark}

The following elementary consequence of the functional calculus allows
us to transfer this interpretation to functions of $M_{n,z}$, in particular
to its centered and rescaled versions. 
\begin{lemma}\label{lem:functional_calculus}
Let $T$ be a self-adjoint operator, $U$ a unitary
 operator, and $h : \R \to \R$ a Borel measurable function. 
We denote by $\mu_{\psi}^T$ the spectral measure of a
 self-adjoint operator $T$ in the state $\psi$. 
If $X \sim \mu_{\psi}^T$, then $h(X) \sim
 \mu_{\psi}^{h(T)}$. Moreover, $\mu_{U\psi}^T = \mu_{\psi}^{U^*TU}$. 
\end{lemma}

We will use the following fact in later sections. 

\begin{lemma}\label{lem:displacement-matrix2}
For $n\in \N_0$ and $z \in \C$, 
\[
\E[M_{n,z}]=n+|z|^2, \quad \Var(M_{n,z})=(2n+1)|z|^2. 
\]
\end{lemma}
\begin{proof}
Since $D(z)$ is unitary, by
 Remark~\ref{rem:spectral-measure-Mnz}, the moments of
 $M_{n,z}$ are given by 
\[
\mathbb E[M_{n,z}^k]
=
\langle D(z)e_n, \num^k D(z)e_n\rangle 
=
\langle e_n, \big( D(z)^* \num D(z)\big)^k e_n\rangle \qquad
 (k \in \N). 
\]
Using the standard intertwining relations
\[
D(z)^*aD(z)=a+z,
\qquad
D(z)^*a^*D(z)=a^*+\zbar,
\]
by writing $T:=za^* + \zbar a$, we obtain
\begin{align}
D(z)^* \num D(z)
&=
(a^*+\zbar)(a+z)
=
\num +za^*+ \zbar a+|z|^2 = \num + T+|z|^2, \label{eq:DND} \\
\big( D(z)^* \num D(z) \big)^2
&=
\num^2+T^2+|z|^4+2|z|^2 \num 
+\num T+T \num +2|z|^2T 
\nonumber.
\end{align}
Since $T$ changes the level by $\pm 1$, we have 
$\langle e_n,\num Te_n\rangle
= \langle e_n,T \num e_n\rangle
= \langle e_n,T e_n\rangle
= 0$, 
and it follows that
\begin{align*}
\mathbb E[M_{n,z}]
&= \langle e_n, D(z)^* \num D(z) e_n\rangle
= \langle e_n, \num e_n\rangle+|z|^2
=
n+|z|^2, \\
\E[M_{n,z}^2]
&=
\langle e_n, \bigl(D(z)^* \num D(z)\bigr)^2e_n\rangle
=
n^2+2|z|^2n+|z|^4+\langle e_n,T^2 e_n\rangle 
=
\E[M_{n,z}]^2 + \langle e_n,T^2 e_n\rangle.
\end{align*}
Since $T^2=z^2(a^*)^2 + \zbar^2 a^2+|z|^2(a^*a+aa^*)$,
 we obtain 
\[
\langle e_n,T^2e_n\rangle
=
|z|^2\bigl(\langle e_n,a^*ae_n\rangle+\langle e_n,aa^*e_n\rangle\bigr)
=
(2n+1)|z|^2.
\]
Thus, $\Var(M_{n,z})=(2n+1)|z|^2$. This completes the proof.
\end{proof}

We note that the process $\bigl(X_n(z)\bigr)_{n\ge 0}$ is covariant
under translations up to a 
deterministic phase factor reflecting the underlying complex
structure. 

\begin{lemma}\label{lem:translation-invariance}
For every $z_0 \in \C$, 
\begin{equation*}
\bigl(X_n(z+z_0)\bigr)_{n\ge 0}
\inlaw 
\bigl(\lambda(z_0,z) X_n(z)\bigr)_{n\ge 0}, 
\end{equation*}
where 
$\lambda(z_0,z) = e^{- i \Im (z_0\bar z)}$. 
\end{lemma}
\begin{proof}
Since $U(z)_{m,n} = e^{-|z|^2/2} H_{m,n}(z,\bar z)$ by
 Lemma~\ref{lem:displacement-matrix},
 $X_n(z)=\sum_{k=0}^{\infty} \zeta_k U(z)_{k,n}$. 
If we regard $\zeta=(\zeta_k)_{k\ge 0}$ as a row vector,
 then 
$\bigl(X_n(z)\bigr)_{n\ge 0} = \zeta U(z)$. 
The identities are understood componentwise, or
 equivalently at the level of finite-dimensional
 distributions. 
The Weyl relations give
\[
D(z_0+z)
=
e^{- \frac{1}{2}(z_0\bar z-\bar z_0 z)} D(z_0)D(z),
\]
hence
\[
U(z_0+z)
=
\lambda(z_0,z)\, U(z_0)U(z),
\qquad
|\lambda(z_0,z)|=1.
\]
Since $U(z_0)$ is unitary, $\widetilde{\zeta}:=\zeta U(z_0) \inlaw \zeta$,
 and hence
\[
\bigl(X_n(z_0+z)\bigr)_{n\ge 0} 
=
\lambda(z_0,z)\, \zeta U(z_0)U(z) 
\inlaw 
\lambda(z_0,z)\, \widetilde{\zeta} U(z)  
\inlaw 
\lambda(z_0,z)\, \bigl(X_n(z)\bigr)_{n\ge 0}. 
\]
This completes the proof. 
\end{proof}

\subsection{Zeros are simple and isolated}\label{subsec:isolated}

For $n=0$, the zeros of 
$f_0(z)$ are isolated, since any
realization of (\ref{GEF}) is an entire function and, as such, it cannot
vanish on an accumulation point without vanishing everywhere
in $\C$. However, for $n\geq 1$, a general function of the form (\ref{poly}) may
have both isolated and non-isolated zeros 
(for example, $z-z^{2}\zbar$ vanishes on a
non-discrete set); see \cite{PolyanalyticZeros} and the references
therein for research about zeros of polyanalytic functions). Nevertheless,
we can show that the zeros of $f_{n}(z,\zbar)$ are
almost surely simple.

Let $U\subset\C\simeq\R^2$ be open and let $g:U\to\C$
be a $C^1$ function. A zero $z_0\in U$ of $g$ is called \emph{simple} (or 
\emph{regular}) if the Jacobian matrix of the map
$G: U\to\R^2$, defined by $G(z) = (\Re g(z),\Im g(z))$, is non-singular at $z_0$, i.e. 
\begin{equation*}
g(z_0)=0\quad\text{and}\quad \det DG(z_0)\neq 0.
\end{equation*}
In this case, $z_0$ is an isolated zero of $g$ by the inverse function
theorem. If $g$ is holomorphic, then this notion agrees with the usual one: $
z_0$ is a simple zero if and only if $g^{\prime }(z_0)\neq 0$.

In our setting, the weak Bulinskaya lemma in
\cite[Proposition~2.2]{AAL25} with $D=d=2$ can be stated as
follows. 

\begin{lemma}\label{lem:bulinskaya}
Let $U\subset \R^2$ be an open set, and let $F:U\to \R^2$
be a random field with almost surely $C^1$ sample paths. Fix $v\in \R^2$.
Assume that for every relatively compact open set $O \subset U$, there exist a
neighborhood $V$ of $v$ and a constant $C_O<\infty$ such that, for every
$y\in V$, 
\begin{equation*}
\int_O p_{F(x)}(y)\,dx \le C_O,
\label{eq:estimate_heat_kernel}
\end{equation*}
where $p_{F(x)}$ denotes the density of $F(x)$ on $\R^2$.
Then
\[
\P\bigl(\exists x\in O:\ F(x)=v,\ \det DF(x)=0\bigr)=0.
\]
In other words, with probability one, the level $v$ is not a
 critical value of $F$. 
\end{lemma}

\begin{proposition}
For every $n \ge 0$, almost surely, all
zeros of the real-analytic map $z\mapsto f_n(z,\overline z)$ are simple in
the sense above. In particular,  for every $n \ge 0$, $\cZ_{f_n}$ forms a point
 process.
\end{proposition}
\begin{proof}
Fix $n \ge 0$ and write the associated $\R^2$-valued field
as 
\begin{equation*}
F_n(z):=\bigl(\Re f_n(z,\overline z),\,\Im f_n(z,\overline
 z)\bigr) \in \R^2, \qquad
z\in\C\simeq\R^2.
\end{equation*}
The sample paths of $F_n$ are real-analytic, since $f_0$ is entire and $f_n=(n!)^{-1/2}(\partial_z-\overline z)^n f_0$ is obtained by applying a
finite-order differential operator with smooth coefficients.
Fix a relatively compact open set $O \subset \C$. For each $z\in O$, the random
variable $f_n(z,\bar z)$ is a centered complex Gaussian with variance
\[
\E\bigl[|f_n(z,\bar z)|^2\bigr]=K_n(z,z)=e^{|z|^2} \ge 1.
\]
Therefore, identifying $y=(y_1,y_2)\in \R^2$ with
 $y_1+iy_2\in \C$, $F_n(z)$ has density
\[
p_{F_n(z)}(y)
=
\frac{1}{\pi K_n(z,z)}
\exp\Bigl(-\frac{|y|^2}{K_n(z,z)}\Bigr) \le \frac{1}{\pi},
\qquad y\in \R^2,  
\]
for all $z\in O$ and all $y\in \R^2$. Hence
\[
\int_O p_{F_n(z)}(y)\,dz \le \frac{|O|}{\pi}
\qquad\text{for all } y\in \R^2.
\]
Lemma~\ref{lem:bulinskaya}, applied with $v=0$, implies that almost surely
there is no point $z\in O$ such that
$F_n(z)=0$ and $\det DF_n(z)=0$.
This means that almost surely every zero of $f_n$ in $O$ is
non-degenerate,
hence simple. By the inverse function theorem, every such
zero is isolated. 
Taking the exhaustion $\{O_m=B(0,m)\}_{m \ge 1}$, 
we conclude that all zeros in $\C$
 are simple almost surely. Since $\cZ_{f_n}$ is closed and has
 no accumulation point in $\C$, its intersection with every
 compact set is finite. Hence $\cZ_{f_n}$ is a locally finite
 simple point process. 
\end{proof} 

\section{Proof of Theorems~\ref{Main} and \ref{LongRange}}\label{sec:proof12}

\subsection{Kac-Rice formulas}

The formulas giving the statistics of zeros of Gaussian functions are known
in general as Kac-Rice formulas. We use a Kac-Rice formula
from \cite{level} after making a few adaptations to fit our
setting and notation, 
so that we can write it in the form (\ref{Kak-Rice}). Let 
\begin{equation*}
\partial_{z}=(\partial _{x}-i\partial _{y})/2,\quad
 \partial_{\zbar} 
=(\partial _{x}+i\partial _{y})/2, 
\label{eq:wirtinger}
\end{equation*}
so that 
\begin{equation*}
\partial_{x}=\partial _{z}+\partial _{\zbar},\quad \partial
_{y}=i(\partial _{z}-\partial _{\zbar}). 
\end{equation*}
For
$f(z)=f(z,\zbar)=u(z,\zbar)+iv(z,\zbar)$, 
we also write $f(z)=(u(z,\zbar),v(z,\zbar))^{\top }\in
\R^{2}$. 
To ease notation, we will sometimes write 
$f(z)$ for $f(z,\zbar)$ when no confusion can arise. 
Then, 
\begin{equation*}
\det f^{\prime }(z)=
\begin{vmatrix}
\partial_{x}u & \partial_{y}u \\ 
\partial_{x}v & \partial_{y}v
\end{vmatrix}
=-2i
\begin{vmatrix}
\partial_{z}u & \partial_{\zbar}u \\ 
\partial_{z}v & \partial_{\zbar}v
\end{vmatrix}
=
\begin{vmatrix}
\partial_{z}f & \partial_{\zbar}f \\ 
\partial_{z}\overline{f} & \partial_{\zbar}\overline{f}
\end{vmatrix}
=
\begin{vmatrix}
\partial _{z}f & \partial _{\zbar}f \\ 
\overline{\partial _{\zbar}f} & \overline{\partial _{z}f}
\end{vmatrix}
=|\partial _{z}f|^{2}-|\partial _{\zbar}f|^{2}. 
\end{equation*}
Since $u(z,\zbar)=v(z,\zbar)=0$ for $z\in
\cZ_{f}:=\{z\in \C:f(z,\zbar)=0\}$,
then 
\begin{equation*}
\nablazup u(z,\zbar)=(\partial_{z}-\zbar)u(z,\zbar)
=\partial_{z}u(z,\zbar)\quad \text{for $z\in \cZ_{f}$.}
\end{equation*}
Hence, 
\begin{equation*}
\det f^{\prime }(z)=|\nablazup f|^{2}-|\nablazdown f|^{2} 
\quad \text{for $z\in \cZ_{f}$.}
\end{equation*}
Assume that $(f,g)$ is a jointly smooth Gaussian process
and that, for the given pair $(z,w)$ with $z \not=w$,  
the random vector $(f(z), g(w))$ is non-degenerate, 
so that its covariance matrix is invertible. 
Then, by the Kac-Rice formula \cite[Exercise 6.1]{level}, 
\begin{align}
\rho _{\cZ_{f},\cZ_g}^{(2)}(z,w) & =\E\left[
|\det f^{\prime }(z)\det g^{\prime }(w)| \ | \
f(z)=g(w)=0\right] \cdot p_{(f(z), g(w))}(0,0)  
\label{KacRice} \\
&=\E\left[ \big| |\nablazup f|^{2}-|\nablazdown f|^{2} \big| 
\cdot \big| |\nablawup g|^{2}-|\nablawdown g|^{2} \big| \ | \ f(z)=g(w)=0
\right] \cdot p_{(f(z), g(w))}(0,0)  \notag
\end{align} 
For fixed $z,w \in \C$, we consider the following $6$-dimensional
Gaussian vector  
\begin{equation*}
\bF(z,w):=(f(z),g(w),\nablazup f(z),\nablazdown
 f(z),\nablawup g(w), \nablawdown g(w)) 
=:(\bF^{(1)},\bF^{(2)}),
\end{equation*}
where
$\bF^{(1)}=(f(z),g(w)),\bF^{(2)}
= (\nablazup f(z), \nablazdown f(z),
\nablawup g(w),\nablawdown g(w))$. 
The covariance matrix of $\bF$ is written
as 
\begin{equation*}
\Sigma =
\begin{pmatrix}
U & V \\ 
V^{\ast } & W
\end{pmatrix},
\end{equation*}
where $U,W$ are the covariance matrix of $\bF^{(1)}$
and $\bF^{(2)}$, respectively, and $V$ is the
cross-covariance matrix of $\bF^{(1)}$ and
$\bF^{(2)}$. Then, the covariance matrix of
$\bF^{(2)}$ given that $\bF^{(1)}=\mathbf{0}$ is provided by the Schur
complement 
\begin{equation*}
W-V^{\ast }U^{-1}V. 
\end{equation*}

\subsection{Conditional Gaussian processes and the joint density}

We encode all derivatives appearing in the Kac-Rice formula into a
6-dimensional Gaussian vector, and then pass to the Schur complement. 
Recall the representation (\ref{FnHermite}):
\begin{equation*}
f_{n}(z,\zbar)=\sum_{k=0}^{\infty }\zeta
 _{k}H_{k,n}(z,\zbar). 
\end{equation*}
From (\ref{ComplexHermite}), one can check that $\nablazup
H_{k,n}(z,\zbar) = \sqrt{n+1}H_{k,n+1}(z,\zbar)$
and that $\nablazdown H_{k,n}(z,\zbar) = -\sqrt{n}H_{k,n-1}(z,\zbar)$. Therefore, 
\begin{eqnarray*}
\nablazup f_{n}(z,\zbar) &=&\sqrt{n+1}\sum_{k=0}^{\infty
}\zeta _{k}H_{k,n+1}(z,\zbar)=\sqrt{n+1}f_{n+1}(z,\zbar) \\
\nablazdown f_{n}(z,\zbar) &=&-\sqrt{n}
\sum_{k=0}^{\infty }\zeta
_{k}H_{k,n-1}(z,\zbar)=-\sqrt{n}f_{n-1}(z,\zbar),
\end{eqnarray*}
leading to the following relations 
\begin{equation}
\nablazup f_{n}=\sqrt{n+1}f_{n+1}, \quad 
\nablazdown f_{n}=-\sqrt{n}f_{n-1}  
\label{eq:lowering-raizing}
\end{equation}
with the convention $f_{-1}=0$. 
We will use these relations systematically to express the entries of 
the covariance matrix in a form suitable for computing both
the short-range and long-range limits, 
without relying on the explicit formulas for the entries in terms
of special functions, which are quite complicated expressions in terms of
Laguerre polynomials, practically untreatable without resorting to symbolic
computation. This will be used both in the proofs of
Theorems~\ref{Main} and \ref{LongRange}. 
Let us consider the following $6$-dimensional Gaussian vector 
\begin{align*}
\bF_{k}(z,w)& = \big(
 f_{n}(z),(\nablawup)^{k}f_{n}(w), \nablazup f_{n}(z),
 \nablazdown f_{n}(z),(\nablawup)^{k+1}f_{n}(w),\nablawdown(\nablawup)^{k}f_{n}(w) \big) \\
& =\big( f_{n}(z),\sqrt{(n+1)_{k}}f_{n+k}(w), \\
&
 \sqrt{n+1}f_{n+1}(z),-\sqrt{n}f_{n-1}(z),\sqrt{(n+1)_{k+1}}f_{n+k+1}(w),
-\sqrt{n+k}\sqrt{(n+1)_{k}}f_{n+k-1}(w) \big),
\end{align*}
where $(\alpha )_{k}=\prod_{j=0}^{k-1}(\alpha +j)$ is the Pochhammer symbol.
We denote the covariance matrix of $\bF_k(z,w)$ by 
$\Sigma_{n,k}(z,w)$ and write it as 
\begin{equation*}
\Sigma_{n,k}(z,w)
=\E\left[ \bF_{k}(z,w)^{\intercal }
\overline{\bF_{k}(z,w)}\right] =
\begin{pmatrix}
U_{n,k}(z,w) & V_{n,k}(z,w) \\ 
V_{n,k}^{\ast }(z,w) & W_{n,k}(z,w), 
\end{pmatrix}
\end{equation*}
where $U_{n,k}(z,w)$ is the covariance matrix of
\begin{equation*}
\bF_{k}^{(1)}(z,w)=(f_{n}(z),(\nablawup)^{k}f_{n}(w))
 = (f_{n}(z),\sqrt{(n+1)_{k}}f_{n+k}(w))
\end{equation*}
and $W_{n,k}(z,w)$ is the covariance matrix of
\begin{equation*}
\bF_{k}^{(2)}(z,w)=(\nablazup f_{n}(z),\nablazdown
 f_{n}(z), (\nablawup)^{k+1}f_{n}(w),\nablawdown(\nablawup)^{k}f_{n}(w))
\end{equation*}
\begin{equation*}
=(\sqrt{n+1}f_{n+1}(z),-\sqrt{n}f_{n-1}(z),\sqrt{(n+1)_{k+1}}f_{n+k+1}(w),
-\sqrt{n+k}\sqrt{(n+1)_{k}}f_{n+k-1}(w)). 
\end{equation*}
Then, setting
\begin{equation}
s_{n,l}=s_{n,l}(z,w):=\E[f_{n}(z,\zbar)\overline{f_{l}(w,
\wbar)}], 
\label{eq:snl-def}
\end{equation}
so that the covariance kernel is
$s_{n,n}=K_{n}(z,w)=L_{n}(|z-w|^{2})e^{z \wbar}$
\begin{align}
U_{n,k}(z,w)& =
\begin{pmatrix}
e^{|z|^{2}} & \sqrt{(n+1)_{k}}s_{n,n+k} \\ 
\sqrt{(n+1)_{k}}\overline{s_{n,n+k}} & (n+1)_{k}e^{|w|^{2}}
\end{pmatrix}
\label{U} \\
\frac{V_{n,k}(z,w)}{\sqrt{(n+1)_{k}}}& =
\begin{pmatrix}
0 & 0 & \sqrt{n+k+1}s_{n,n+k+1} & -\sqrt{n+k}s_{n,n+k-1} \\ 
\sqrt{n+1}\overline{s_{n+1,n+k}} & -\sqrt{n}\overline{s_{n-1,n+k}} & 0 & 0
\end{pmatrix}
\label{V}
\end{align}
\begin{equation*}
W_{n,k}(z,w)=
\end{equation*}
\begin{equation}
\begin{pmatrix}
(n+1)e^{|z|^{2}} & 0 & \sqrt{(n+1)}\sqrt{(n+1)_{k+1}} \overline{s_{n+1,n+k+1}
} & -\sqrt{(n+1)(n+k)(n+1)_{k}} \overline{s_{n+1,n+k-1}} \\ 
0 & ne^{|z|^{2}} & -\sqrt{(n)_{k+2}}s_{n-1,n+k+1} & \sqrt{(n)_{k+1}(n+k)}
s_{n-1,n+k-1} \\* 
& \ast & (n+1)_{k+1}e^{|w|^{2}} & 0 \\* 
& \ast & 0 & (n+1)_{k}(n+k)e^{|w|^{2}}
\end{pmatrix}
\label{w}
\end{equation}
where we used that $K_{n}(z,z)=e^{|z|^{2}}$ for every $n=0,1,\dots $ , which
follows from $L_{n}(0)=1$.
Note that for $z\ne w$, the two evaluation functionals appearing in
$\bF_k^{(1)}(z,w)$ are linearly independent; hence
$U_{n,k}(z,w)$ is non-singular.

The Schur complement is the covariance matrix of the
Gaussian vector 
$\bF_{k}^{(2)}(z,w)$ given that $\bF_{k}^{(1)}(z,w)=(0,0)$,
i.e. $f_{n}(z)=(\nablawup)^{k}f_{n}(w)=0$, which is calculated
as 
\begin{equation*}
\Lambda_{n,k}(z,w):=W_{n,k}(z,w)-V_{n,k}^{\ast
}(z,w)U_{n,k}^{-1}(z,w)V_{n,k}(z,w).
\end{equation*}
$\Lambda_{0,1}$ is the same as Feng's $\Lambda $ in \cite{Feng}. The joint
density of $(\bF_k^{(1)}(z,w), \bF_{k}^{(2)}(z,w))$ at $((0,0),
\xi)$ can now be written as 
\begin{equation*}
p(\xi )=\frac{e^{-\left\langle \Lambda _{n,k}^{-1}(z,w)\xi ,\xi
\right\rangle }}{\det \pi U_{n,k}(z,w)\det \pi \Lambda_{n,k}(z,w)}, 
\end{equation*}
where $\xi =(\xi _{1},\xi _{2},\xi _{3},\xi _{4})$.

\subsection{Proof of Theorem~\ref{Main}}\label{subsec:proofoftheorem1}

Combining the proposition below with (\ref{onepoint}) we obtain the precise
values for the correlations as $w\rightarrow z$: 
\begin{equation}
g_{n,n+k}(z,z) = 
\begin{cases}
\disp
 \frac{n(n+2)}{(n+\frac{1}{2}+\frac{1}{4n+2})(n+\frac{3}{2}+\frac{1}{4n+6})}<1
 & \text{if $k=1 $} \\[4mm]  
 \disp \frac{n^2+3n+1 + \frac{2}{3(n^{2}+3n+1)}\left\{ \frac{n^{4}(n+2)^{2}}{(2n+1)^{2}
 }+\frac{(n+1)^{2}(n+3)^{4}}{(2n+5)^{2}}\right\} }{(n+\frac{1}{2}+\frac{1}{4n+2}){(n+\frac{5}{2}+\frac{1}{4n+10})}}>1 & \text{if $k=2$} \\[4mm] 
\disp 1 & \text{if $k\geq 3$},  
\end{cases}
\label{precise}
\end{equation}
where $g_{n,n+k}(z,z)$ is defined as in \eqref{eq:gzz}. 

\begin{proposition}
Let $\rho_{n,n+k}^{(2)}(z,z) = \lim_{w \to z : w\not=z}
 \rho_{n,n+k}^{(2)}(z,w)$. Then, 
\begin{equation*}
\rho_{n,n+k}^{(2)}(z,z) = 
\begin{cases}
\disp \frac{n(n+2)}{\pi^{2}} & \text{if $k=1$} \\[2mm] 
 \disp \frac{1}{\pi^{2}}\left[ n^2+3n+1 + \frac{2}{3(n^2+3n+1)} 
 \left\{ \frac{n^{4}(n+2)^{2}}{(2n+1)^{2}}+\frac{(n+1)^{2}(n+3)^{4}}{(2n+5)^{2}}\right\} 
 \right] & \text{if $k=2$} \\[2mm] 
\disp \rho _{n}^{(1)}(z)\rho _{n+k}^{(1)}(z) & \text{if $k\geq
 3$}. 
\end{cases}
\end{equation*}
\end{proposition}

\begin{proof}
We first observe from \eqref{eq:discrete-orthogonality} and
 \eqref{eq:snl-def} that 
\begin{equation*}
\lim_{w\rightarrow z}s_{n,l}(z,w) 
= e^{|z|^2} \delta_{n,l}. 
\end{equation*}
For $k\geq 1$, as $w\rightarrow z$, we have 
\begin{align*}
U_{n,k}(z,z)& =e^{|z|^{2}}
\begin{pmatrix}
1 & 0 \\ 
0 & (n+1)_{k}
\end{pmatrix}
\\
V_{n,k}(z,z)& =e^{|z|^{2}}
\begin{pmatrix}
0 & 0 & 0 & -(n+1)\delta _{k,1} \\ 
(n+1)\delta _{k,1} & 0 & 0 & 0
\end{pmatrix}
\\
W_{n,k}(z,z)& =e^{|z|^{2}}
\begin{pmatrix}
n+1 & 0 & 0 & -(n+1)_{k}\delta _{k,2} \\ 
0 & n & 0 & 0 \\ 
0 & 0 & (n+1)_{k+1} & 0 \\ 
-(n+1)_{k}\delta _{k,2} & 0 & 0 & (n+1)_{k}(n+k)
\end{pmatrix}
\end{align*}
where $\delta _{a,b}$ is Kronecker's delta. Therefore, since $V_{n,k}(z,z)=0$
for $k\geq 2$, 
\begin{equation}
\Lambda_{n,k}(z,z)=e^{|z|^{2}}
\begin{cases}
\begin{pmatrix}
0 & 0 & 0 & 0 \\ 
0 & n & 0 & 0 \\ 
0 & 0 & (n+1)(n+2) & 0 \\ 
0 & 0 & 0 & 0
\end{pmatrix}
& \text{if $k=1$} \\ 
W_{n,k} & \text{if $k\geq 2$}
\end{cases}
\label{eq:lambdank}
\end{equation}

Thus, from (\ref{KacRice}), the correlation between the zero
 sets of $f_{n}(z,\zbar)$ and $f_{n+k}(w,\wbar)$ is given by 
\begin{align*}
\rho _{n,n+k}^{(2)}(z,w) 
&= \E\left[ \left\vert
|\nablazup f_{n}|^{2}-|\nablazdown f_{n}|^{2}\right\vert
\cdot \big| |\nablawup(\nablawup)^{k}f_{n}|^{2}
-|\nablawdown(\nablawup)^{k}f_{n}|^{2} \big| \ | \
f_{n}(z)=h_{n+k}(w)=0\right] \\
&\quad \times p_{(f_n(z), h_{n+k}(w))}(0,0), 
\end{align*}
where $h_{n+k}(w) = (\nablawup)^k f_n(w) = \sqrt{(n+1)_k}
 f_{n+k}(w)$.  
Since $h_{n,k}$ is a non-zero scalar multiple of $f_{n+k}$,
 this gives the same zero intensity for the pair
 $(\cZ_{f_n}, \cZ_{f_{n+k}})$. 
Passing to the limit as $w \to z$, we obtain 
\begin{eqnarray*}
\rho_{n,n+k}^{(2)}(z,w) &=&\frac{1}{\det \pi U_{n,k}(z,w)}\int_{\C
^{4}}\left\vert |\xi _{1}|^{2}-|\xi _{2}|^{2}\right\vert \cdot \left\vert
|\xi _{3}|^{2}-|\xi _{4}|^{2}\right\vert \frac{e^{-\left\langle \Lambda
_{n,k}^{-1}(z,w)\xi ,\xi \right\rangle }}{\det \pi \Lambda _{n,k}(z,w)}
dV_{\xi } \\
&\rightarrow &\frac{1}{\det \pi U_{n,k}(z,z)}\E\left[ \left\vert
|X_{1}|^{2}-|X_{2}|^{2}\right\vert \cdot \left\vert
|X_{3}|^{2}-|X_{4}|^{2}\right\vert \right]
\end{eqnarray*}
where $(X_{1},X_{2},X_{3},X_{4})$ is a centered complex Gaussian vector with covariance
matrix $\Lambda_{n,k}(z,z)$ given by (\ref{eq:lambdank}). 
Although $\Lambda_{n,k}(z,z)$ is singular, the passage to
 the limit is justified as follows. The conditional Gaussian
 vectors with covariance $\Lambda_{n,k}(z,w)$ converge in
 distribution to the centered Gaussian vector with
 covariance $\Lambda_{n,k}(z,z)$. Since the integrand is a
 polynomially bounded function of the Gaussian vector and
 the covariance matrices remain bounded near $w=z$, the
 corresponding family is uniformly integrable. Hence the
 expectations converge.
Note that 
\begin{equation*}
\det \pi U_{n,k}(z,z)=\pi^{2} e^{2|z|^2} (n+1)_{k}.
\end{equation*}
Let $Z_{1},Z_{2},Z_{3},Z_{4}$ be i.i.d. standard complex normal random
variables. Then, from \eqref{eq:lambdank}, we see that 
\begin{equation}
(X_{1},X_{2},X_{3},X_{4}) \inlaw e^{|z|^2/2}
\begin{cases}
(0,\sqrt{n}Z_{2},\sqrt{(n+1)_{2}}Z_{3},0) & \text{for $k=1$}, \\[2mm] 
(\sqrt{n+1}Z_{1}, \sqrt{n}Z_{2}, \sqrt{(n+1)_{3}}Z_{3}, -\sqrt{(n+1)_{2}(n+2)}
 Z_{1}) & \text{for $k=2$}, \\[2mm] 
(\sqrt{n+1}Z_{1}, \sqrt{n}Z_{2}, \sqrt{(n+1)_{k+1}}Z_{3}, \sqrt{(n+1)_{k}(n+k)}
Z_{4}) & \text{for $k\geq 3$}. 
\end{cases}
\label{eq:inlaw}
\end{equation}
When $k=1$, using \eqref{eq:inlaw} and the formula \eqref{eq:formula1} in
Lemma~\ref{lem:formula1}, 
\begin{align*}
\rho _{n,n+1}^{(2)}(z,w)& \rightarrow \frac{1}{\det \pi
 U_{n,1}(z,z)}\E 
\left[ \left\vert |X_{1}|^{2}-|X_{2}|^{2}\right\vert \cdot \left\vert
|X_{3}|^{2}-|X_{4}|^{2}\right\vert \right] \\
& =\frac{1}{\pi ^{2}(n+1)}\E\left[ \left\vert
-n|Z_{2}|^{2}\right\vert \right] \cdot \E\left[ \left\vert
(n+1)_{2}|Z_{3}|^{2}\right\vert \right] \\
& =\frac{1}{\pi ^{2}(n+1)}n\cdot (n+1)_{2} \\
& =\frac{1}{\pi ^{2}}n(n+2).
\end{align*}
When $k\geq 3$, using \eqref{eq:inlaw} and the formula \eqref{eq:formula1}
in Lemma~\ref{lem:formula1}, as $w\rightarrow z$ 
\begin{align*}
\rho _{n,n+k}^{(2)}(z,w)& \rightarrow \frac{1}{\det \pi U_{n,k}(z,z)}\mathbb{
E}\left[ \left\vert |X_{1}|^{2}-|X_{2}|^{2}\right\vert \cdot \left\vert
|X_{3}|^{2}-|X_{4}|^{2}\right\vert \right] \\
& =\frac{1}{\pi ^{2}(n+1)_{k}}\E\left[ \left\vert
(n+1)|Z_{1}|^{2}-n|Z_{2}|^{2}\right\vert \right] \cdot \E\left[
\left\vert (n+1)_{k+1}|Z_{3}|^{2}-(n+k)(n+1)_{k}|Z_{4}|^{2}\right\vert 
\right] \\
& =\frac{1}{\pi
 ^{2}(n+1)_{k}}\frac{(n+1)^{2}+n^{2}}{(n+1)+n}\cdot 
\frac{(n+1)_{k+1}^{2}+(n+k)^{2}(n+1)_{k}^{2}}{(n+1)_{k+1}+(n+k)(n+1)_{k}} \\
& =\frac{1}{\pi ^{2}}\frac{(n+1)^{2}+n^{2}}{(n+1)+n}\cdot 
\frac{(n+k+1)^{2}+(n+k)^{2}}{(n+k+1)+(n+k)} \\
& =\rho _{n}^{(1)}(z)\rho _{n+k}^{(1)}(w)
\end{align*}
When $k=2$, using \eqref{eq:inlaw}, we note that as $w\rightarrow z$ 
\begin{align*}
\rho _{n,n+2}^{(2)}(z,w)& \rightarrow \frac{1}{\det \pi
 U_{n,2}(z,z)}
\E\left[ \left\vert |X_{1}|^{2}-|X_{2}|^{2}\right\vert \cdot \left\vert
|X_{3}|^{2}-|X_{4}|^{2}\right\vert \right] \\
& =\frac{1}{\pi ^{2}(n+1)_{2}}\E\left[ \left\vert
(n+1)|Z_{1}|^{2}-n|Z_{2}|^{2}\right\vert \cdot \left\vert
(n+1)_{3}|Z_3|^{2} - (n+2)(n+1)_{2}|Z_1|^2 \right\vert \right] .
\end{align*}
We first assume $n\geq 1$. The case $n=0$ follows either by
 a direct computation or by taking the continuous limit
 $b \downarrow 0$ in Lemma~\ref{lem:formula1}. 
Using the formula \eqref{eq:formula2} in Lemma~\ref{lem:formula1} with $%
x=(n+1)/n,\ y=(n+2)/(n+3),\ bd/(n+1)_{2}=n(n+3)$, we have 
\begin{align*}
& 2xy-(x+y)+1=\frac{n^2+3n+1}{n(n+3)}, \\
&
 \frac{2y^{2}}{(1+x)^{2}(1+x+y)}=
\frac{2n^3(n+2)^2}{3(n+3)(2n+1)^2 (n^2+3n+1)}, \\ 
& \frac{2x^{2}}{(1+y)^{2}(1+x+y)}
= \frac{2(n+3)^3(n+1)^2}{3n (2n+5)^2 (n^2+3n+1)}. 
\end{align*} 
Therefore, 
\begin{equation*}
\rho _{n,n+2}^{(2)}(z,w)\rightarrow \frac{1}{\pi ^{2}}\left[ 
n^2+3n+1
+\frac{2}{3(n^2+3n+1)}\left\{
\frac{n^4(n+2)^2}{(2n+1)^2}
+ \frac{(n+1)^2(n+3)^4}{(2n+5)^2} 
\right\} \right]. 
\end{equation*}
This completes the proof. 
\end{proof}

The next lemma is used to compute the expectations in the
above proof. 

\begin{lemma}\label{lem:formula1}
Let $Z_{1},Z_{2},Z_{3}$ be i.i.d. standard complex normal random variables.
For $a,b,c,d>0$, 
\begin{equation}
\E\left[ \left\vert a|Z_{1}|^{2}-b|Z_{2}|^{2}\right\vert \right] 
= \frac{a^{2}+b^{2}}{a+b}. 
\label{eq:formula1}
\end{equation}
\begin{align}
\lefteqn{\E\left[ |a|Z_{1}|^{2}-b|Z_{2}|^{2}|\cdot
|c|Z_{1}|^{2}-d|Z_{3}|^{2}|\right] }  \notag \\
& =bd\left( 2xy-(x+y)+\frac{2y^{2}}{(1+x)^{2}(1+x+y)}+\frac{2x^{2}}{(1+y)^{2}(1+x+y)}+1\right), 
\label{eq:formula2}
\end{align}
where $x=a/b$ and $y=c/d$.
\end{lemma}

\begin{proof}
First we note that if $Z$ is a standard complex normal random variable, then 
$\left\vert Z\right\vert^{2}\sim \mathrm{Exp}(1)$. Now, for $\fe
\sim \mathrm{Exp}(1)$ and $u>0$, 
\begin{equation*}
\E[\left\vert u-\fe\right\vert ]=u+2e^{-u}-1
\end{equation*}
and
\begin{equation*}
\E[e^{-u \fe}]=\frac{1}{1+u},\quad \E[\fe e^{-u \fe}]=
\frac{1}{(1+u)^{2}}. 
\end{equation*}

Let $\fe_{i}\sim \mathrm{Exp}(1)$ and i.i.d. below. First we see that
\begin{equation*}
\E[\left\vert x\fe_{1}-\fe_{2}\right\vert ]=
\E\left[ \E[\left\vert x\fe_{1}-\fe
_{2}\right\vert \fe_{1}\right] ]=\E[x\fe
_{1}+2e^{-x\fe_{1}}-1 ]=\frac{1+x^{2}}{1+x}. 
\end{equation*}
Therefore, setting $x=a/b$, we obtain 
\begin{equation*}
\E\left[ \left\vert a|Z_{1}|^{2}-b|Z_{2}|^{2}\right\vert \right] =b
\E\left[ |\frac{a}{b}\fe_{1}-\fe_{2}|\right] =
\frac{a^{2}+b^{2}}{a+b}.
\end{equation*}

Set $x=a/b$ and $y=c/d$. Since
 $|x\fe_{1}-\fe_{2}|$ 
and $|y\fe_{1}-\fe_{3}|$ are independent
 given $\fe_{1}$, we see that 
\begin{align*} 
\lefteqn{ \E[\left\vert x\fe_{1}-\fe_{2}\right\vert \cdot 
\left\vert y\fe_{1}-\fe_{3}\right\vert ]} \\ 
&= \E\Big[ 
\E\big[\left| x\fe_{1}-\fe_{2}\right| \big| \fe_{1}\big] 
\cdot \E\big[\left| y \fe_{1}-\fe_{3}\right| \big| \fe_{1} \big] 
\Big] \\
&= \E\big[ \{x\fe_{1}+2e^{-x\fe_{1}}-1\}\{y
\fe_{1}+2e^{-y\fe_{1}}-1\}  \big] \\
&=\E\big[ xy\fe_{1}^{2}-(x+y)\fe_{1}+2(y
\fe_{1}e^{-x\fe_{1}}+x\fe_{1}e^{-y\fe
_{1}})+(4e^{-(x+y)\fe_{1}}-2(e^{-x\fe_{1}}+e^{-y \fe_1})+1) \big]\\
&=2xy-(x+y)+2y\frac{1}{(1+x)^{2}}+2x\frac{1}{(1+y)^{2}}+\left( 4\frac{1}{
1+x+y}-2\frac{1}{1+x}-2\frac{1}{1+y}+1\right) \\
&=2xy-(x+y)+\frac{2y^{2}}{(1+x)^{2}(1+x+y)}+\frac{2x^{2}}{(1+y)^{2}(1+x+y)}
+1. 
\end{align*}
Therefore, we have \eqref{eq:formula2}.
\end{proof}

\begin{remark}\label{rem:4/3}
As noted after  Remark~\ref{rem:corr=1}, the value $4/3$
 also appears as the saddle-point density of $f_0(z)$ in \cite{Maxima}. 
In the Landau-level Kac--Rice computation in this section,
the $k=2$ contact limit is the first case where the two Jacobian
factors share one intermediate Landau level. 
Up to the sign of one Jacobian factor, which is immaterial after taking
absolute values, the high-level limit ($n \to \infty$) gives
\[
\E\left[|\fe_1-\fe_2|\,|\fe_1-\fe_3|\right]=\frac{4}{3},
 \qquad \E\left[(\fe_1-\fe_2)(\fe_1-\fe_3)\right]=1.
\]
On the other hand, the saddle-point density in \cite{Maxima} is also obtained from a
sign-split exponential Jacobian integral; in the same normalization
one encounters
\[
\E\left[(2\fe_1-\fe_2)_+\right]=\frac{4}{3}, \qquad
\E\left[2\fe_1-\fe_2\right]=1.
\]
Thus, in both cases, a signed Jacobian expectation
equal to $1$ is converted into $4/3$ by an absolute-value or
positive-part operation. This suggests a common elementary
exponential structure behind the two computations, although a
direct geometric interpretation is not yet clear.
\end{remark}

\subsection{Proof of Theorem~\ref{LongRange}}

The result follows from the proposition below, by observing
that (\ref{onepoint}) can be written as 
\begin{equation*}
\rho_{n}^{(1)}(z)=\frac{1}{\pi}\frac{(n+1)^{2}+n^{2}}{2n+1}. 
\end{equation*}

\begin{proposition}
As $|z-w| \to \infty $  
\begin{equation*}
\rho_{n,n+k}^{(2)}(z,w) \to \frac{1}{\pi ^{2}}\frac{(n+1)^{2}+n^{2}}{
(n+1)+n}\cdot \frac{(n+k+1)^{2}+(n+k)^{2}}{(n+k+1)+(n+k)}. 
\end{equation*}
\end{proposition}

\begin{proof}
For $n,l=0,1,\dots $, let 
\begin{equation*}
s_{n,l}(z,w):=\E[f_{n}(z,\zbar)\overline{f_{l}(w,\wbar)
}].
\end{equation*}
Since 
\begin{equation*}
f_{l}(w,\wbar)=\sum_{p=0}^{\infty }\zeta _{p}H_{p,l}(w,\wbar)
\end{equation*}
we see that, since $H_{p,l}(0,0)=\delta _{p,l}$, 
\begin{equation*}
f_{l}(0,0)=\sum_{p=0}^{\infty }\zeta _{p}H_{p,l}(0,0)=\zeta _{l}.
\end{equation*}
Then, 
\begin{equation*}
s_{n,l}(z,0)=\E[f_{n}(z,\zbar)\overline{f_{l}(0,0)}]=\mathbb{E%
}\left[ \sum_{p=0}^{\infty }\zeta _{p}H_{p,n}(z,\zbar)\overline{\zeta
_{l}}\right] =H_{l,n}(z,\zbar).
\end{equation*}
Setting $w=0$ in (\ref{U})-(\ref{w}) we see that 
\begin{align*}
U_{n,k}(z,0)& =
\begin{pmatrix}
e^{|z|^{2}} & \sqrt{(n+1)_{k}}H_{n+k,n} \\ 
\sqrt{(n+1)_{k}}H_{n+k,n}^{\ast } & (n+1)_{k}
\end{pmatrix}
\\
\frac{V_{n,k}(z,0)}{\sqrt{(n+1)_{k}}}& =
\begin{pmatrix}
0 & 0 & \sqrt{n+k+1}H_{n+k+1,n} & -\sqrt{n+k}H_{n+k-1,n} \\ 
\sqrt{n+1}H_{n+k,n+1}^{\ast } & -\sqrt{n}H_{n+k,n-1}^{\ast } & 0 & 0
\end{pmatrix}
\end{align*}
\begin{align*}
\lefteqn{W_{n,k}(z,0)} \\
&= 
\begin{pmatrix}
(n+1)e^{|z|^{2}} & 0 & \sqrt{(n+1)}\sqrt{(n+1)_{k+1}}H_{n+k+1,n+1} & -\sqrt{%
(n+1)(n+k)(n+1)_{k}}H_{n+k-1,n+1} \\ 
0 & ne^{|z|^{2}} & -\sqrt{(n)_{k+2}}H_{n+k+1,n-1} & \sqrt{(n)_{k+1}(n+k)}%
H_{n+k-1,n-1} \\* 
& \ast & (n+1)_{k+1} & 0 \\* 
& \ast & 0 & (n+1)_{k}(n+k)%
\end{pmatrix}%
,
\end{align*}
where $H_{k}^{\ast
 }(z,\zbar)=\overline{H_{k}(z,\zbar)}$. 
By translation covariance and rotational invariance, it is enough to consider $(z,w)=(r,0)$.
Set $z=r$ and, as in the proof of Theorem~\ref{Main}, define 
\begin{equation*}
\Lambda_{n,k}(r)=W_{n,k}(r)-V_{n,k}(r)^{\ast }U_{n,k}(r)^{-1}V_{n,k}(r)
\end{equation*}
so that
\begin{equation*}
\rho _{n,n+k}^{(2)}(r,0)=\frac{1}{\det \pi U_{n,k}(r)}\int_{\C
^{4}}\left\vert |\xi _{1}|^{2}-|\xi _{2}|^{2}\right\vert \cdot \left\vert
|\xi _{3}|^{2}-|\xi _{4}|^{2}\right\vert \frac{e^{-\left\langle \Lambda
_{n,k}^{-1}(r)\xi ,\xi \right\rangle }}{\det \pi \Lambda _{n,k}(r)}dV_{\xi }.
\end{equation*}
We rewrite this using the matrices
\begin{equation*}
P=
\begin{bmatrix}
1 & 0 & 0 & 0 \\ 
0 & -1 & 0 & 0 \\ 
0 & 0 & 0 & 0 \\ 
0 & 0 & 0 & 0
\end{bmatrix}
,\quad 
Q=
\begin{bmatrix}
0 & 0 & 0 & 0 \\ 
0 & 0 & 0 & 0 \\ 
0 & 0 & 1 & 0 \\ 
0 & 0 & 0 & -1
\end{bmatrix}
\end{equation*}
and the change of variable $\eta =\Lambda _{n,k}^{-\frac{1}{2}}(r)\xi $\ to
obtain 
\begin{equation*}
\rho_{n,n+k}^{(2)}(r,0)
=\frac{1}{\det \pi U_{n,k}(r)}\int_{\C^4} 
\Big| \langle \eta,
 \Lambda_{n,k}^{1/2}P\Lambda_{n,k}^{1/2}\eta \rangle \cdot
 \langle \eta, \Lambda _{n,k}^{1/2}Q\Lambda _{n,k}^{1/2}\eta
 \rangle\Big| \frac{e^{-|\eta |^{2}}}{\pi ^{4}}dV_{\eta} . 
\end{equation*}
The off-diagonal Hermite terms have at most polynomial growth in $r$, whereas the leading
diagonal terms are of size $e^{r^2}$. Consequently,
$\det U_{n,k}(r,0)\sim (n+1)_k e^{r^2}$,
and the conditional covariance matrix becomes asymptotically
 diagonal. Then we obtain 
\begin{align*}
\lefteqn{\rho_{n,n+k}^{(2)}(r,0)} \\  
& \sim \frac{1}{\pi
 ^{2}(n+1)_{k}e^{r^{2}}}\int_{\C^{4}}\big| 
(n+1)e^{r^{2}}|\eta _{1}|^{2}-ne^{r^{2}}|\eta _{2}|^{2}\big|\cdot
\big| (n+1)_{k+1}|\eta _{3}|^{2}-(n+1)_{k}(n+k)|\eta
_{4}|^{2}\big| \frac{e^{-|\eta |^{2}}}{\pi ^{4}}dV_{\eta } \\
& =\frac{1}{\pi ^{2}}\int_{\C^{4}}\left\vert (n+1)|\eta
_{1}|^{2}-n|\eta _{2}|^{2}\right\vert \left\vert (n+k+1)|\eta
_{3}|^{2}-(n+k)|\eta _{4}|^{2}\right\vert \frac{e^{-|\eta |^{2}}}{\pi ^{4}}%
dV_{\eta } \\
& =\frac{1}{\pi ^{2}}\E\left[ |(n+1)|Z_{1}|^{2}-n|Z_{2}|^{2}|\cdot
|(n+k+1)|Z_{3}|^{2}-(n+k)|Z_{4}|^{2}|\right] \\
&=\frac{1}{\pi ^{2}}\frac{(n+1)^{2}+n^{2}}{(n+1)+n}\cdot 
\frac{(n+k+1)^{2}+(n+k)^{2}}{(n+k+1)+(n+k)} \\
&=\rho _{n}^{(1)}(z)\rho _{n+k}^{(1)}(w) 
\end{align*}
by Lemma~\ref{lem:formula1} applied with $a=n+1$, 
$b=n$, $c=n+k+1$, $d=n+k$. 
\end{proof}


\section{Proof of Theorem~\protect\ref{thm:functionalLLN}}\label{sec:proof34}
\subsection{Pointwise law of large numbers and central limit
theorem}\label{subsec:pointwise}

We first compute the covariance structure of $\{f_n(z,\bar
z)\}_{n = 0}^{\infty}$. 
\begin{lemma}\label{lem:covariance_z_and_w} 
For $n, n' \in \N_0$ and $z,w \in \C$, 
\begin{equation}
\E \left[ f_{n}(z,\bar{z})\overline{f_{n^{\prime }}(w,\wbar)}
\right] =(-1)^{n^{\prime }+n}H_{n^{\prime
},n}(w-z,\wbar-\zbar) e^{z\wbar}. 
\label{eq:different_index}
\end{equation}
In particular, $\{e^{-|z|^{2}/2}f_{k}(z,\bar{z})\}_{k=0}^{\infty }$ are
i.i.d. complex Gaussian random variables with zero mean and
unit variance. 
\end{lemma}
\begin{proof}
Write $\tilde{H}_{p,q}(z,\zbar)$ for the complex Hermite polynomial
in \cite{Ghanmi}. The relationship between $\tilde{H}_{p,q}(z,\zbar)$
and the $H_{p,q}(z,\zbar)$ that we have been using in this paper is
given as 
\begin{equation*}
H_{p,q}(z,\zbar)=\frac{1}{\sqrt{p!q!}}\tilde{H}_{p,q}(z,\zbar).
\end{equation*}
We note that 
\begin{equation*}
H_{p,q}(z,\zbar)=(-1)^{p-q}H_{q,p}(\zbar,z). 
\label{eq:Hsymmetric}
\end{equation*}
From \cite[Proposition 3.7]{Ghanmi}, 
\begin{equation*}
\sum_{k=0}^{\infty
 }\frac{\tilde{H}_{l,k}(z,\zbar)\overline{\tilde{H}_{l^{\prime},k}(w,\wbar)}}{k!}=\tilde{H}_{l,l^{\prime}}(z-w,
 \zbar-\wbar)e^{\zbar w}=\sqrt{l!l^{\prime}!} 
H_{l,l^{\prime}}(z-w,\zbar-\wbar)e^{\zbar w}. 
\end{equation*}
Then, 
\begin{align*}
\E \left[ f_{n}(z,\bar{z})\overline{f_{n^{\prime}}(w,\wbar)}
\right] 
&= \sum_{k=0}^{\infty}H_{k,n}(z,\bar{z})\overline{H_{k,n^{\prime}}(w,\wbar)} \\  
&= \sum_{k=0}^{\infty}(-1)^{k-n}H_{n,k}(\zbar,z) \cdot 
\overline{(-1)^{k-n^{\prime}} H_{n^{\prime},k}(\wbar,w)} \\
& =\sum_{k=0}^{\infty }(-1)^{n^{\prime}+n}H_{n^{\prime},k}(w,\wbar)
\overline{H_{n,k}(z,\bar{z})}  \\ 
& =\frac{(-1)^{n^{\prime }+n}}{\sqrt{n!(n^{\prime
 })!}}\sum_{k=0}^{\infty } 
\frac{\tilde{H}_{n^{\prime
 },k}(w,\wbar)\overline{\tilde{H}_{n,k}(z,\bar{z})}}{k!}
 \nonumber\\ 
& =(-1)^{n^{\prime }+n}H_{n^{\prime
 },n}(w-z,\wbar-\zbar)e^{\wbar z}. 
\end{align*}
Since $H_{p,q}(0,0)=\delta_{p,q}$, we conclude that 
\begin{equation*}
\E\left[ e^{-|z|^{2}}f_{n}(z,\bar{z})\overline{f_{n^{\prime}}(z,
\zbar)}\right] =\delta_{n^{\prime},n}. 
\end{equation*}
This completes the proof. 
\end{proof}

By \eqref{lem:covariance_z_and_w}, for fixed $z\in \C$, 
$\big\{ \big|
e^{-|z|^{2}/2}f_{k}(z,\zbar)\big|^{2} \big\}_{k=0}^{\infty}$ 
are i.i.d. with $| e^{-|z|^{2}/2}f_{k}(z,\bar{z})|^{2}\sim 
\mathrm{Exp}(1)$. 
By the strong law of large numbers, for each $z \in \C$, 
\begin{equation*}
\frac{1}{N} \sum_{k=0}^{N-1} | e^{-|z|^{2}/2}f_{k}(z,\zbar)
 |^{2} \rightarrow 1\text{ a.s. }
\end{equation*}
Moreover, by the classical central limit theorem, 
\begin{equation*}
\frac{\sum_{k=0}^{N-1} 
\big(| e^{-|z|^{2}/2}f_{k}(z,\bar{z}) |^{2}-1\big)}{\sqrt{N}}
\inlawto N(0,1). 
\end{equation*}

\subsection{Functional law of large numbers} 

We observed the pointwise law of large numbers, which implies the
convergence in the finite-dimensional distribution to the 
constant process $1$ on $\C$. 
We now strengthen the pointwise almost sure convergence to
almost sure convergence on compact sets. 
In fact, the argument yields a stronger statement than
Theorem~\ref{thm:functionalLLN}, namely convergence in
$C^{0,\alpha}(K)$ (see the definition below) for every compact set $K \subset \C$. 
The key point of the proof is to obtain a
uniform $L^p$-estimate on the $\alpha$-H\"older seminorm of $S_N$. 
To this end, we first recall a Morrey-type estimate for
H\"older seminorms. 

For a fixed compact set $K \subset \R^d$ and $\alpha \in
(0,1)$,  
\[
C^{0,\alpha}(K)
:=
\left\{
u\in C(K): [u]_{C^{0, \alpha}(K)}<\infty
\right\}.
\]
where 
\[
[u]_{C^{0,\alpha}(K)}
:=
\sup_{\substack{z,w\in K\\ z\neq w}}
\frac{|u(z)-u(w)|}{|z-w|^{\alpha}}.
\]
Then $C^{0,\alpha}(K)$ is a Banach space of H\"older continuous
functions on $K$ with norm
\[
\|u\|_{C^{0,\alpha}(K)}
:=
\|u\|_{L^\infty(K)}+[u]_{C^{0,\alpha}(K)}.
\]
We will use the following consequence of Morrey's
inequality \cite[Corollary~9.14]{Brezis}. 
Take a bounded Lipschitz domain $U \subset \R^d$ such that
$K\subset U$. 
Choose $p>d$ so that $\alpha < 1-d/p$, 
and set $\beta = 1-d/p \in (\alpha, 1)$. 
By Morrey's inequality on bounded $C^1$ domains, 
\[
[u]_{C^{0,\beta}(K)} \le [u]_{C^{0,\beta}(\overline{U})} \le C
\|u\|_{W^{1,p}(U)}. 
\]
Applying this to $u-u_U$, where $u_U:=|U|^{-1}\int_U u$, and
using the Poincar\'e inequality, we obtain 
\begin{equation}
[u]_{C^{0,\beta}(K)}
\le
C_U \|\nabla u\|_{L^p(U)}.
\label{eq:Morrey} 
\end{equation}
Since $\alpha < \beta$, 
\begin{equation*}
[u]_{C^{0,\alpha}(K)} \le (\mathrm{diam} K)^{\beta-\alpha}[u]_{C^{0,\beta}(K)}
\le C\|u\|_{W^{1,p}(U)}. 
\end{equation*}
\begin{lemma}\label{lem:holder}
Let $K \subset \C$ be compact and $0 < \alpha < 1$. Then there
 exist $p>2$ and a constant $C_{K,\alpha}>0$ such that 
\[
\E\left[[S_N]_{C^{0,\alpha}(K)}^p\right] \le \frac{C_{K,\alpha}}{N^{p/2}}. 
\]
\end{lemma}
\begin{proof}
Recall that
\[
X_n(z):=e^{-|z|^2/2}f_n(z,\bar z),
\qquad
S_N(z):=\frac1N\sum_{n=0}^{N-1}|X_n(z)|^2.
\]
By \eqref{eq:lowering-raizing}, 
\[
\partial_z X_n
=
-\frac{\bar z}{2}X_n+e^{-|z|^2/2}\partial_z f_n
=
-\frac{\bar z}{2}X_n+e^{-|z|^2/2}(\nablazup f_n+\bar z f_n)
=
\sqrt{n+1}\,X_{n+1}+\frac{\bar z}{2}X_n,
\]
and similarly
\[
\partial_{\bar z}X_n
=
-\sqrt n\,X_{n-1}-\frac{z}{2} X_n.
\]
Then, 
\[
\partial_z |X_n|^2
=
(\partial_z X_n)\overline{X_n}
+
X_n\,\overline{\partial_{\bar z}X_n}
=
\sqrt{n+1}\,X_{n+1}\overline{X_n}
-
\sqrt n\,X_n\overline{X_{n-1}}. 
\]
Summing over $n=0,1, \dots,N-1$, we obtain the telescoping identities
\begin{equation}
\partial_z S_N(z)
=
\frac{1}{N} \sum_{n=0}^{N-1}\partial_z |X_n(z)|^2
=
\frac{\sqrt N}{N}X_N(z)\overline{X_{N-1}(z)}.  
\label{eq:telescoping} 
\end{equation}
Since $\overline{\partial_z S_N(z)} = \partial_{\zbar} S_N(z)$, 
\[
|\nabla S_N(z)|^2
=
2\bigl(|\partial_z S_N(z)|^2+|\partial_{\bar z}S_N(z)|^2\bigr)
=
\frac{4}{N} |X_N(z)|^2 \cdot |X_{N-1}(z)|^2.
\]
For each fixed $z\in\C$, the random variables
$X_N(z)$ and $X_{N-1}(z)$ are independent standard complex Gaussian. Since
$\E|Z|^p=\Gamma(1+\frac{p}{2})$ for $Z\sim N_{\C}(0,1)$, it follows that
\[
\E |\nabla S_N(z)|^p
=
\frac{2^p}{N^{p/2}}  
\E|X_N(z)|^p \cdot \E|X_{N-1}(z)|^p
=
\frac{C_p}{N^{p/2}},   
\]
where $C_p = 2^p \Gamma(1+\frac{p}{2})^2$. 
Since $S_N$ is smooth, by \eqref{eq:Morrey}, 
\[
[S_N]_{C^{0,\alpha}(K)}
\le
C_{U,K,\alpha,p} \|\nabla S_N\|_{L^p(U)}, 
\]
and hence 
\[
\E\left[ [S_N]_{C^{0,\alpha}(K)}^p \right]
\le
C_{U,K,\alpha,p}^p \,\E\left[\|\nabla S_N\|_{L^p(U)}^p \right]
=
C_{U,K,\alpha,p}^p \,\int_U \E|\nabla S_N(z)|^p dz 
= \frac{C_{K,\alpha}}{N^{p/2}}. 
\]
This proves the assertion. 
\end{proof}

As mentioned in Section~\ref{sec:main}, we can prove a stronger version of
Theorem~\ref{thm:functionalLLN} 
\begin{theorem}\label{thm:holder}
Let
\[
 S_N(z):=\frac1N\sum_{n=0}^{N-1}\big| e^{-|z|^2/2}f_n(z,\bar
 z) \big|^2, \qquad z\in\C.
\]
Then, for every compact $K\subset\C$ and every
 $0<\alpha <1$, 
\[
S_N \to 1 \quad \text{a.s. in $C^{0,\alpha}(K)$}. 
\]
\end{theorem}
\begin{proof}
Fix a compact set $K\subset\C$, $0<\alpha<1$, and choose $z_0\in K$.
By \eqref{eq:1}, we have that $S_N(z_0)\to 1$ a.s. 
On the other hand, Lemma~\ref{lem:holder} implies that for every $\varepsilon>0$,
$\sum_{N=1}^\infty 
\mathbb{P}\!\left([S_N]_{C^{0,\alpha}(K)}>\varepsilon\right) <
\infty$ since $p>2$. 
Hence, by the Borel--Cantelli lemma, we conclude that
$[S_N]_{C^{0,\alpha}(K)} \to 0$ almost surely. 
For every $z\in K$,
\[
|S_N(z)-1|
\le
|S_N(z)-S_N(z_0)| + |S_N(z_0)-1|
\le
[S_N]_{C^{0,\alpha}(K)}\,|z-z_0|^{\alpha} + |S_N(z_0)-1|.
\]
Taking the supremum over $z\in K$, we get
\[
\sup_{z\in K}|S_N(z)-1|
\le
(\operatorname{diam}K)^{\alpha}[S_N]_{C^{0,\alpha}(K)}
+
|S_N(z_0)-1|.
\]
The right-hand side tends to $0$ almost surely, and hence
\[
\|S_N-1\|_{C(K)}\to 0
\quad\text{a.s.} 
\]
Since also $[S_N-1]_{C^{0,\alpha}(K)}=[S_N]_{C^{0,\alpha}(K)}\to
 0$ a.s., 
we obtain
\[
\|S_N-1\|_{C^{0,\alpha}(K)}\to 0
\quad\text{a.s.} 
\]
This proves the theorem.
\end{proof}

\subsection{Remark on fluctuations of $S_N$}\label{subsec:flucuations}
Since \eqref{eq:2} yields a pointwise central limit
 theorem for each fixed 
$z \in \C$, it is natural to ask whether the fluctuation fields
\begin{equation*}
G_N(z):=\sqrt{N} \bigl(S_N(z)-1\bigr) = \frac{1}{\sqrt{N}} 
\sum_{n=0}^{N-1} (|X_n(z)|^2-1) 
\label{eq:GNz} 
\end{equation*}
also satisfy a functional central limit theorem on
 compact sets. 
However, this question is rather
 delicate. Indeed, from \eqref{eq:telescoping}, we see that 
\[
\partial_z G_N(z)=X_N(z)\overline{X_{N-1}(z)},
\]
so the spatial fluctuations of $G_N$ are governed by a
 boundary term involving only the last two levels rather
 than by the same averaging mechanism that yields the
 pointwise central limit theorem. 
This is already reflected in the finite-dimensional
 distributions. 
Indeed, one can show that 
for every fixed $k \in \N$ and $z_1, z_2,\dots,z_k \in \C$, 
\begin{equation*}
\bigl(G_N(z_1),G_N(z_2), \dots, G_N(z_k) \bigr) \inlawto (Z,Z,\dots,Z),
\qquad Z\sim N(0,1). 
\label{eq:degenerated_convergence}
\end{equation*}
Thus, at fixed macroscopic points, the fluctuations collapse
 onto the diagonal, instead of converging to a nondegenerate
 Gaussian process with spatially varying covariance. This
 suggests that any nontrivial functional central limit
 theorem should be sought only after
 spatial smoothing or spatial rescaling. 
We show a functional CLT after spatial rescaling in the next section. 

\section{Functional central limit theorem}\label{sec:FCLT}
\subsection{Limiting distribution of $M_{n,z}$} 

We first identify the one-parameter limits associated with the random variables
$M_{n,z}$ defined in Section~\ref{subsec:orthonormal}.

\begin{proposition}\label{prop:arcsine-limit}
Let $n_N$ be a sequence of integers such that $n_N/N \to t \in (0,1]$.
For fixed $\xi \in \C$, 
\[
Y_{N,\xi}
:=
\frac{M_{n_N,\sqrt{N}\,\xi} - n_N - N|\xi|^2}{N} 
\inlawto A_{t,|\xi|}, 
\]
where $A_{t,r}$ is a random variable with arcsine
 distribution on $[-2\sqrt{t}\,r, 2\sqrt{t}\,r]$, whose density is given
 by 
$\pi^{-1} (4tr^2-y^2)^{-1/2} \one_{\{|y|<2\sqrt{t} r\}} dy$;
if $|\xi|=0$,  $A_{t,0}$ is understood as the degenerated random variable $0$
\end{proposition}
\begin{proof}
By Remark~\ref{rem:spectral-measure-Mnz}, 
the law of $M_{n,z}$ is the spectral measure of the number
 operator $\num$ in the state $D(z)e_n$. 
We note that the spectral measure of an self-adjoint
 operator $T$ in the state $\psi$, 

By Lemma~\ref{lem:functional_calculus}, 
the law of $Y_{N,\xi}$ is the spectral measure, in the
state $e_{n_N}$, of the self-adjoint operator
\[
\cY_{N,\xi}
:=
\frac{D(\sqrt{N}\,\xi)^* \mathcal N D(\sqrt{N}\,\xi)-n_N-N|\xi|^2}{N}.
\]
From \eqref{eq:DND}, we have  
\[
\cY_{N,\xi}
= \frac{1}{N} \big( \num - n_N + \sqrt{n} (\xi a^* + \overline{\xi} a) \big). 
\]
Put $\delta_j^{(N)} := e_{n_N+j}$.
Using \eqref{eq:creation-annihilation}, we get, for every fixed integer $j$, 
\[
\cY_{N,\xi}\delta_j^{(N)}
=
\frac{j}{N}\delta_j^{(N)}
+
\xi \sqrt{\frac{n_N+j+1}{N}}\,\delta_{j+1}^{(N)}
+
\bar\xi \sqrt{\frac{n_N+j}{N}}\,\delta_{j-1}^{(N)}.
\]
On each finite window
 $\operatorname{Span}\{\delta_j^{(N)}: |j|\le p\}$, 
as $N \to \infty$, 
the matrix of $\cY_{N,\xi}$ converges entrywise to the Jacobi operator
\[
J_{t,\xi}\delta_j
=
\xi \sqrt{t}\,\delta_{j+1}
+
\bar\xi \sqrt{t}\,\delta_{j-1},
\qquad
j\in \mathbb Z.
\]
Write $\xi = |\xi| e^{i\phi}$. Under the Fourier transform
$\mathcal F \delta_j(\theta)=e^{ij\theta}$ on
$L^2([0,2\pi], d\theta/(2\pi))$, one has
\[
\mathcal F J_{t,\xi}\mathcal F^{-1}
=
2|\xi|\sqrt{t}\,\cos(\theta+\phi)\cdot,
\qquad
\mathcal F \delta_0 = 1.
\]
Hence the spectral measure of $J_{t,\xi}$ at $\delta_0$ is the push-forward of
$d\theta/(2\pi)$ under
\[
\theta \mapsto 2|\xi|\sqrt{t}\,\cos(\theta+\phi),
\]
namely the arcsine law above. 
Therefore, for every integer $p\ge 1$,
\[
\E[Y_{N,\xi}^p] = \langle e_{n_N}, \cY_{N,\xi}^p e_{n_N}\rangle
\to
\langle \delta_0, J_{t,\xi}^p \delta_0\rangle 
= \E[A_{t,|\xi|}^p]. 
\]
It follows from this fact that the family $\{|Y_{N,\xi}|^p\}_{N \ge 1}$ is uniformly
 integrable for every $p \ge 1$.  
Hence the family $\{Y_{N,\xi}\}_{N\ge 1}$ is tight, and every subsequential limit has
the same moments as the arcsine law, and since the latter has compact support,
it is moment-determinate. Therefore the full sequence
 converges in law to the arcsine law.
\end{proof}

In the next section, we will use the following lemma on
Ces\`{a}ro means for triangular arrays. It is a direct 
application of Is\'eki's theorem on continuous convergence on compact spaces;
see \cite{I57}. 

\begin{lemma}\label{lem:riemann}
Let $\{a_{n,N}\}$ be a triangular array of complex numbers, where
$N\ge 1$ and $0\le n\le N$. Assume that there exists a function
$f:[0,1]\to \C$ such that for every $t\in [0,1]$ and every sequence
$n_N$ with $0\le n_N\le N$ and $n_N/N\to t$, one has
\[
a_{n_N,N}\to f(t).
\]
Then $f$ is continuous on $[0,1]$ and
\[
\frac{1}{N}\sum_{n=0}^{N-1} a_{n,N}
\to
\int_0^1 f(t)\,dt.
\]
\end{lemma}
\begin{proof}
For each $N\ge 1$, define
\[
g_N(t) := a_{\lfloor Nt\rfloor,N},
\qquad
t\in [0,1].
\]
If $t_N\to t$, then $\lfloor Nt_N\rfloor/N \to t$, so by assumption
\[
g_N(t_N) \to f(t).
\]
Thus $g_N$ converges continuously to $f$ on the compact set $[0,1]$.
By Is\'eki's theorem, $f$ is continuous and
\[
\sup_{t\in [0,1]} |g_N(t)-f(t)| \to 0.
\]
Consequently,
\[
\frac{1}{N}\sum_{n=0}^{N-1} a_{n,N}
=
\int_0^1 g_N(t)\,dt
\to
\int_0^1 f(t)\,dt.
\]
\end{proof}

\subsection{Convergence of the covariance}

First, by Lemma~\ref{lem:translation-invariance}, 
taking squared moduli of $\{X_n(z)\}_{n\ge 0}$ and then
forming the averages appearing in the 
definition of $\cG_N^{(z_0)}$, we immediately obtain the invariance 
$\cG_N^{(z_0)} \inlaw \cG_N^{(0)}$. 

From now on we write
\begin{equation}
\cG_N(\xi):=
\sqrt{N}(S_N(\sqrt{N} \xi)
-1) = \frac{1}{\sqrt{N}} \sum_{n=0}^{N-1} (|X_n(\sqrt{N}\xi)|^2 -1),
\label{eq:defofGN} 
\end{equation}
which corresponds to the case $z_0=0$.
We prove the following. 
\begin{proposition}\label{prop:covariance-limit}
For every $\xi,\eta \in \C$,
\[
\Cov\bigl(\cG_N(\xi),\cG_N(\eta)\bigr)
\to
\kappa(|\xi-\eta|), 
\]
where $\kappa(r)$ is the same as in Theorem~\ref{thm:functional-clt}. 
\end{proposition}
\begin{proof}
Since the family $\{X_n(z)\}_{n\ge 0}$ is jointly complex Gaussian,
Wick's formula yields
\[
\Cov\bigl(|X_n(z)|^2-1,\ |X_m(w)|^2-1\bigr)
=
\bigl|
\E\bigl[X_n(z)\overline{X_m(w)}\bigr]
\bigr|^2
=
\P(M_{n,w-z}=m), 
\]
where the second equality follows from \eqref{eq:different_index}.  

Set $u:=\eta-\xi$ and $r:=|u|$. Then, 
\begin{equation}
\Cov\bigl(\cG_N(\xi),\cG_N(\eta)\bigr)
=
\frac{1}{N}
\sum_{n,m=0}^{N-1}
\P\bigl(M_{n,\sqrt{N}\,u}=m\bigr)
=
\frac{1}{N}
\sum_{n=0}^{N-1}
\P\bigl(M_{n,\sqrt{N}\,u}<N\bigr).
\label{eq:covGNandGN} 
\end{equation}
Fix a sequence $n_N$ with $n_N/N\to t\in (0,1]$. 
Since the law of $A_{t,r}$ is absolutely continuous, 
 Proposition~\ref{prop:arcsine-limit} yields 
\[
\P\bigl(M_{n_N,\sqrt{N}\,u}<N\bigr)
=
\P\left(
\frac{M_{n_N,\sqrt{N}\,u}-n_N-Nr^2}{N}
<
1-\frac{n_N}{N}-r^2
\right)
\to
q_r(t),
\]
where
\[
q_r(t):=
\P\bigl(A_{t,r}<1-t-r^2\bigr)
= 
\begin{cases}
1,
&
r+\sqrt{t}\le 1,
\\[1ex]
0,
&
r\ge 1+\sqrt{t},
\\[1ex]
\frac{1}{\pi}
\arccos \frac{t+r^2-1}{2r\sqrt{t}},
&
|r-\sqrt{t}|<1<r+\sqrt{t}.
\end{cases}
\]
For $r=0$ this gives $q_0(t)=1$. Applying
Lemma~\ref{lem:riemann} on each interval $[\delta,1]$ and then letting
$\delta \downarrow 0$, we obtain
\[
\Cov\bigl(\cG_N(\xi),\cG_N(\eta)\bigr)
\to
\int_0^1 q_r(t)\,dt.
\]

It remains to identify the last integral. Let $T$ be uniformly distributed on
$[0,1]$, let $\Theta$ be uniformly distributed on $[0,2\pi]$, and assume that
$T$ and $\Theta$ are independent. Then
$Y:=\sqrt{T}\,e^{i\Theta}$
is uniformly distributed on the unit disk $B(0,1)$ with density $1/\pi$.
For fixed $t\in (0,1]$, the event $Y\in B(re_1,1)$ under the
 condition $\{T=t\}$ is equivalent to
\[
t+r^2-2r\sqrt{t}\cos \Theta < 1,
\]
and therefore
\[
\P\bigl(Y\in B(re_1,1)\mid T=t\bigr)=q_r(t).
\]
Taking expectation in $T$, we conclude that
\[
\int_0^1 q_r(t)\,dt
=
\P\bigl(Y\in B(re_1,1)\bigr)
=
\frac{|B(0,1)\cap B(re_1,1)|}{\pi}
=
\kappa(r).
\]
This proves the proposition.
\end{proof}

\subsection{Convergence of the finite-dimensional distributions}
We next prove the convergence of all finite-dimensional
marginals. 
First we note the following well-known formula. 
\begin{lemma}\label{lem:quadratic-form}
Let $Z$ be a centered complex Gaussian vector in $\C^d$ with covariance matrix
$\Gamma \ge 0$, and let $A\in M_d(\C)$ be Hermitian. Then, for every
$t\in \R$,
\[
\E \exp\bigl(it Z^* A Z\bigr)
=
\det(I-it\Gamma A)^{-1}.
\]
\end{lemma}

The following proposition establishes the finite-dimensional
convergence.

\begin{proposition}\label{prop:fdd}
Fix $m\ge 1$, $\xi_1,\dots,\xi_m \in \C$, and $a_1,\dots,a_m \in \R$. Then
\[
\sum_{j=1}^m a_j \cG_N(\xi_j)
\inlawto N(0,\sigma^2),
\]
where
\[
\sigma^2
=
\sum_{j,k=1}^m a_j a_k \kappa(|\xi_j-\xi_k|).
\]
Consequently, the finite-dimensional distributions of $\cG_N$ converge to those of
the centered real Gaussian field with covariance $\kappa(|\xi-\eta|)$.
\end{proposition}

\begin{proof}
For $1\le j\le m$ and $0\le n\le N-1$, write
\[
Z_{j,n}^{(N)} := X_n(\sqrt{N}\,\xi_j),
\]
and collect these variables into the column vector
\[
Z^{(N)} := \bigl(Z_{j,n}^{(N)}\bigr)_{1\le j\le m,\ 0\le n\le N-1} \in \C^{mN}.
\]
Let $\Gamma^{(N)}$ be its covariance matrix and its 
$(j,k)$-block component of size $N$ is given by 
\begin{equation}
\Gamma^{(N)}_{jk} 
:=
\Big(\E\bigl[Z_{j,n}^{(N)} \overline{Z_{k,\ell}^{(N)}}\bigr]
 \Big)_{0 \le n \le N-1, 0 \le \ell \le N-1} \qquad (1 \le
 j,k \le m).
\label{eq:Gamma-block} 
\end{equation}
and set $A:=\diag(a_1 I_N,\dots,a_m I_N)$. 
Since the diagonal entries of $\Gamma^{(N)}$ are all equal to $1$, 
\[
L_N:=\sum_{j=1}^m a_j \cG_N(\xi_j)
=
\frac{1}{\sqrt{N}}
\Bigl(
(Z^{(N)})^* A Z^{(N)} - \Tr(\Gamma^{(N)} A)
\Bigr).
\]

We claim that
\[
\|\Gamma^{(N)}\| \le m
\qquad
\text{for all } N\ge 1.
\]
Let $v=(v_1,\dots,v_m)^T \in \C^{mN}$, where each $v_j =
 (v_{j,n})_{n=0}^{N-1}\in \C^N$ is a row vector, and define 
$\alpha_j := \sum_{n=0}^{N-1} v_{j,n} Z^{(N)}_{j,n}$. Then 
\[
(v^*\Gamma^{(N)} v)^{1/2}
= \big\|\sum_{j=1}^m \alpha_j \big\|_{L^2(\P)}
\le 
\sum_{j=1}^m \big\|\alpha_j \big\|_{L^2(\P)}
= 
\sum_{j=1}^m \big\|v_j \big\|_{\C^N}
\le \sqrt{m} \big\|v \big\|_{\C^{mN}}. 
\]
Thus $\| \Gamma^{(N)} \| \le m$, and therefore
\begin{equation}
\|\Gamma^{(N)} A\| \le m \max_{1\le j\le m}|a_j| =: M.
\label{eq:normofGNA} 
\end{equation}

Let $\phi_N(t):=\E e^{itL_N} \ (t\in \R)$. 
By Lemma~\ref{lem:quadratic-form},
\[
\phi_N(t)
=
\exp\left(
-\frac{it}{\sqrt{N}} \Tr(\Gamma^{(N)} A)
\right)
\det\left(
I-\frac{it}{\sqrt{N}}\Gamma^{(N)} A
\right)^{-1}.
\]
For all sufficiently large $N$, its logarithm can be expanded as 
\[
\log \phi_N(t)
=
\sum_{p=2}^{\infty}
\frac{(it)^p}{pN^{p/2}}
\Tr\bigl((\Gamma^{(N)} A)^p\bigr).
\]

We first identify the quadratic term. By 
\eqref{eq:defofGN} and \eqref{eq:Gamma-block},  
\[
\frac{1}{N}\Tr\bigl((\Gamma^{(N)} A)^2\bigr)
=
\sum_{j,k=1}^m a_j a_k
\frac{1}{N}\Tr\bigl(\Gamma^{(N)}_{jk}\Gamma^{(N)}_{kj}\bigr)
= 
\sum_{j,k=1}^m a_j a_k
\Cov\bigl(\cG_N(\xi_j),\cG_N(\xi_k)\bigr).
\]
Hence Proposition~\ref{prop:covariance-limit} gives
\[
\frac{1}{N}\Tr\bigl((\Gamma^{(N)} A)^2\bigr)
\to
\sum_{j,k=1}^m a_j a_k \kappa(|\xi_j-\xi_k|)
=:
\sigma^2.
\]

Next we estimate the higher-order terms. For every $p\ge 3$, 
by \eqref{eq:normofGNA}, 
\[
\left|
\Tr\bigl((\Gamma^{(N)} A)^p\bigr)
\right|
\le
mN\, \|\Gamma^{(N)} A\|^p
\le
mN\, M^p. 
\]
Therefore, for sufficiently large $N$, 
\[
\sum_{p=3}^{\infty}
\frac{|t|^p}{pN^{p/2}}
\left|
\Tr\bigl((\Gamma^{(N)} A)^p\bigr)
\right|
\le
mN
\sum_{p=3}^{\infty}
\frac{1}{p}
\left(
\frac{|t|M}{\sqrt{N}}
\right)^p, 
\]
which converges to $0$ as $N \to \infty$. Consequently,
\[
\log \phi_N(t)
\to
-\frac{t^2}{2}\sigma^2.
\]
The convergence of the finite-dimensional distributions of $\cG_N$ follows from the
 Cram\'er--Wold device. 
\end{proof}

\subsection{Tightness in H\"older spaces}

We now prove the increment bounds needed for tightness.

\begin{proposition}\label{prop:l2-increment}
There exists a universal constant $C<\infty$ such that, for all $N\ge 1$ and
all $\xi,\eta\in \C$,
\[
\E\bigl[|\cG_N(\xi)-\cG_N(\eta)|^2\bigr]
\le
C |\xi-\eta|.
\]
\end{proposition}
\begin{proof}
Set $u:=\eta-\xi$ and $r:=|u|$. If $r=0$, there is nothing to prove.
By \eqref{eq:covGNandGN},  
\[
\E\bigl[|\cG_N(\xi)-\cG_N(\eta)|^2\bigr]
=
2\bigl(1-\Cov(\cG_N(\xi),\cG_N(\eta))\bigr)
=
\frac{2}{N}
\sum_{n=0}^{N-1}
\P(M_{n,\sqrt{N}\,u}\ge N).
\]
If $r\ge 1$, then trivially
\[
\E\bigl[|\cG_N(\xi)-\cG_N(\eta)|^2\bigr]
\le 2
\le 2r.
\]
Thus it remains to consider the case $0<r<1$.

Write the sum in terms of $j=N-n$ and use the trivial bound
 for $1 \le j \le 2Nr$ to obtain 
\[
\E\bigl[|\cG_N(\xi)-\cG_N(\eta)|^2\bigr]
=
\frac{2}{N}\sum_{j=1}^N \P(M_{N-j,\sqrt{N}\,u}\ge N)
\le 4r + \frac{2}{N}\sum_{j > 2Nr} \P(M_{N-j,\sqrt{N}\,u}\ge N)
\]
Fix $j>2Nr$. By Lemma~\ref{lem:displacement-matrix2},
since $0<r<1$, 
\[
N-\E[M_{N-j,\sqrt{N}\,u}]
=
j-Nr^2
\ge
\frac{j}{2}, 
\]
and 
\[
\Var(M_{N-j,\sqrt{N}\,u})
=
(2(N-j)+1)Nr^2
\le
3N^2 r^2.
\]
By Chebyshev's inequality,
\[
\P(M_{N-j,\sqrt{N}\,u}\ge N)
\le
\frac{\Var(M_{N-j,\sqrt{N}\,u})}
{\bigl(N-\E[M_{N-j,\sqrt{N}\,u}]\bigr)^2}
\le
\frac{12 N^2 r^2}{j^2}, 
\]
and hence 
\[
\frac{2}{N}\sum_{j>2Nr} \P(M_{N-j,\sqrt{N}\,u}\ge N) 
\le 24 Nr^2 \sum_{j>2Nr}\frac{1}{j^2}
\le c Nr^2 \min\left\{1,\frac{1}{Nr}\right\}
\le c r. 
\]
Therefore, 
\[
\E\bigl[|\cG_N(\xi)-\cG_N(\eta)|^2\bigr]
\le
Cr
=
C|\xi-\eta|.
\]
\end{proof}

We recall Nelson's hypercontractive estimate for Gaussian
chaos; see \cite[Chapter 6]{Janson1997}.
\begin{lemma}\label{lem:hypercontractive}
If $F$ belongs to the second homogeneous Gaussian chaos, then for every $p\ge 2$,
\[
\|F\|_{L^p}
\le
(p-1)\|F\|_{L^2}.
\]
\end{lemma}

Using this, we upgrade the $L^2$ increment estimate to
higher moments.

\begin{proposition}\label{prop:lp-increment}
For every $p\ge 2$ there exists $C_p<\infty$ such that, for all $N\ge 1$ and
all $\xi,\eta\in \C$,
\[
\E\bigl[|\cG_N(\xi)-\cG_N(\eta)|^p\bigr]
\le
C_p |\xi-\eta|^{p/2}.
\]
\end{proposition}

\begin{proof}
Fix $N\ge 1$ and $\xi,\eta\in \C$, and set
\[
F_{N,\xi,\eta}:=\cG_N(\xi)-\cG_N(\eta).
\]
For each $z\in \C$ and each $n\ge 0$, the random variable $X_n(\sqrt{N}\,z)$ is
a centered complex Gaussian linear functional of the underlying Gaussian
sequence $(\zeta_k)_{k\ge 0}$. Therefore, $|X_n(\sqrt{N}\,z)|^2-1$
is a centered quadratic polynomial in the real Gaussian family
$(\Re \zeta_k,\Im \zeta_k)_{k\ge 0}$. Hence $F_{N,\xi,\eta}$ belongs to the
second homogeneous Gaussian chaos. By
Lemma~\ref{lem:hypercontractive} and
Proposition~\ref{prop:l2-increment},
\[
\|F_{N,\xi,\eta}\|_{L^p}
\le
(p-1)\|F_{N,\xi,\eta}\|_{L^2}
\le
C_p |\xi-\eta|^{1/2}.
\]
Raising both sides to the power $p$ gives the claim.
\end{proof}

The higher moment bound now gives tightness by the
Kolmogorov--Chentsov theorem.

\begin{proposition}\label{prop:tightness-holder}
Fix a compact set $K\subset \C$ and $0<\alpha<\frac{1}{2}$. Then the family
$\{\cG_N\}_{N\ge 1}$ is tight in $C^{0,\alpha}(K)$. In particular, it is tight in
$C(K)$.
\end{proposition}

\begin{proof}
Choose $\beta$ such that $\alpha<\beta<1/2$. 
Next choose $p>4$ so large that
\[
\beta < \frac{1}{2}-\frac{2}{p},
\qquad
\text{equivalently}\qquad
\frac{p}{2}>2+\beta p.
\]
By Proposition~\ref{prop:lp-increment},
\[
\sup_{N\ge 1}
\E\bigl[|\cG_N(\xi)-\cG_N(\eta)|^p\bigr]
\le
C_p |\xi-\eta|^{p/2},
\qquad
\xi,\eta\in K.
\]
Since $K\subset \C \simeq \R^2$ is two-dimensional and
$\frac{p}{2}>2+\beta p$, the Kolmogorov--Chentsov theorem yields
\[
\sup_{N\ge 1}
\E\bigl[[\cG_N]_{C^{0,\beta}(K)}^p\bigr]
\le
C_{K,p,\beta}
\]
for some constant $C_{K,p,\beta}<\infty$; see for instance 
\cite{Billingsley1999, Kallenberg2021}. 

Fix $\xi_0\in K$. By (7.1), $\E[|\cG_N(\xi_0)|^2]=1$, and
$\cG_N(\xi_0)$ belongs to the second homogeneous Gaussian chaos. Hence
Lemma~\ref{lem:hypercontractive} gives
$\sup_{N\ge 1} \E\bigl[|\cG_N(\xi_0)|^p\bigr] < \infty$.
Using
\[
\|f\|_{C^{0,\beta}(K)}
\le
|f(\xi_0)| + \bigl(1+(\operatorname{diam} K)^\beta\bigr)[f]_{C^{0,\beta}(K)},
\]
we conclude that
\[
\sup_{N\ge 1}\E\bigl[\|\cG_N\|_{C^{0,\beta}(K)}^p\bigr] < \infty.
\]

For $R>0$, let
\[
B_R
:=
\{f\in C^{0,\beta}(K): \|f\|_{C^{0,\beta}(K)}\le R\}.
\]
Since $\beta>\alpha$ and $K$ is compact, the embedding
$C^{0,\beta}(K)\hookrightarrow C^{0,\alpha}(K)$ 
is compact. Therefore each $B_R$ is compact in $C^{0,\alpha}(K)$.
By Markov's inequality,
\[
\sup_{N\ge 1}\P(\cG_N\notin B_R)
\le
\frac{1}{R^p}
\sup_{N\ge 1}\E\bigl[\|\cG_N\|_{C^{0,\beta}(K)}^p\bigr].
\]
The right-hand side tends to $0$ as $R\to\infty$, which proves tightness in
$C^{0,\alpha}(K)$. Since the inclusion
$C^{0,\alpha}(K)\hookrightarrow C(K)$ is continuous, tightness in $C(K)$ follows
as well.
\end{proof}

\begin{remark}
 Regularity of the limiting field $\cG$ can also be
 verified in the same manner. 
\end{remark}

\subsection{Regularity of the limiting field}
We now construct the limiting Gaussian field as a continuous process.
\begin{proposition}\label{prop:holder-version-limit}
There exists a centered real Gaussian process $\cG$ on $\C$ with covariance
\[
\E\bigl[\cG(\xi)\cG(\eta)\bigr] = \kappa(|\xi-\eta|),
\]
and for every compact set $K\subset \C$ and every $0<\alpha<\frac{1}{2}$, this
process admits a version with sample paths in $C^{0,\alpha}(K)$.
\end{proposition}

\begin{proof}
By Proposition~\ref{prop:covariance-limit}, for every finite set
$\{\xi_1,\dots,\xi_m\} \subset \C$, the matrix
 $\bigl(\kappa(|\xi_j-\xi_k|)\bigr)_{1\le j,k\le m}$ 
is positive semidefinite as a pointwise limit of covariance matrices. Hence,
by Kolmogorov's extension theorem, there exists a centered real Gaussian process
$\cG$ with the stated covariance.

Since
\[
\E\bigl[|\cG(\xi)-\cG(\eta)|^2\bigr]
=
2\bigl(1-\kappa(|\xi-\eta|)\bigr),
\]
it is enough to bound $1-\kappa(r)$ by a constant multiple of $r$. For
$0\le r\le 2$,
\[
1-\kappa(r)
=
\frac{2}{\pi}\arcsin\frac{r}{2}
+
\frac{r}{\pi}\sqrt{1-\frac{r^2}{4}}
\le
Cr,
\]
while for $r\ge 2$ the same estimate is trivial because $1-\kappa(r)=1$.
Therefore
\[
\E\bigl[|\cG(\xi)-\cG(\eta)|^2\bigr]
\le
C|\xi-\eta|.
\]
Since $\cG(\xi)-\cG(\eta)$ is Gaussian, for every $p\ge 2$,
\[
\E\bigl[|\cG(\xi)-\cG(\eta)|^p\bigr]
\le
C_p |\xi-\eta|^{p/2}.
\]
Choosing $p>4$ and applying Kolmogorov--Chentsov on $K\subset \R^2$, we obtain a
version with sample paths in $C^{0,\alpha}(K)$ for every
$0<\alpha<\frac{1}{2}$.
\end{proof}

\begin{proof}[Proof of Theorem~\ref{thm:functional-clt}]
By Lemma~\ref{lem:translation-invariance}, it is enough to treat the case $z_0=0$.
By Proposition~\ref{prop:fdd}, the finite-dimensional distributions of $\cG_N$
converge to those of the centered Gaussian field $\cG$ from
Proposition~\ref{prop:holder-version-limit}. By
Proposition~\ref{prop:tightness-holder}, the sequence of laws of $\{\cG_N\}_{N\ge 1}$
is tight in $C(K)$. Since $C(K)$ is a Polish space, Prokhorov's theorem applies,
and every subsequence admits a further subsequence converging weakly in $C(K)$.

Let $\mu$ be any subsequential limit. The coordinate maps
$f\mapsto f(\xi)$ are continuous on $C(K)$, hence $\mu$ has the same
finite-dimensional distributions as $\cG$. Because a Borel probability measure on
$C(K)$ is determined by its finite-dimensional marginals, $\mu$ must be the law
of $\cG$. Thus every subsequential limit is the same, and therefore
\[
\cG_N \inlawto \cG \qquad \text{in } C(K).
\]
The additional H\"older statements are exactly
Propositions~\ref{prop:tightness-holder} and
\ref{prop:holder-version-limit}.
\end{proof}

\section{White noise spectrograms}\label{sec:whitenoise}
\subsection{Gaussian white noise}

\label{sec:WN}One can formally define complex white noise $\cW$ by
the series expansion 
\begin{equation}
\cW(t)=\sum_{k=0}^{\infty }\zeta _{k}h_{k}(t), 
\label{wn}
\end{equation}
where $\zeta _{k}\sim N_{\C}(0,1)$ i.i.d. However, 
such a series is almost surely not square summable. 
To make sense of the expansion \eqref{wn}, one may either
work in a larger weighted space as in \cite[Section 3]{BH} 
or start by defining Brownian motion \cite{AbrAlpayJorTry}. 
We will use a more direct approach based on the construction of the white
noise space, as in \cite{levelSets}. 
Let $\cS(\R)$ denote the
Schwartz space of rapidly decreasing smooth functions from
$\R$ to $\C$, and
let $\cS^{\prime}(\R)$ be its continuous dual,  
the space of tempered distributions consisting of 
all continuous linear functionals on $\cS(\R)$. 
Since $\cS(\R)$ is dense in $L^{2}(\R)$ and
$\cS(\R)\subseteq L^{2}(\R)\subseteq
\cS^{\prime }(\R)$, this defines a
\emph{Gelfand triple} $(\cS(\R),L^{2}\left(
\R\right) ,\cS^{\prime}(\R))$. The 
duality pairing between $\cS(\R)$ and
$\cS^{\prime }(\R)$ is defined by $\langle
\psi ,\phi \rangle :=\psi (\phi )$, for $\psi \in
\cS^{\prime}(\R)$ and $\phi \in
\cS(\R)$. 

The space $\cS^{\prime}(\R)$ is the basic
sample space. Denote by
$\mathcal{B}(\cS^{\prime}(\R))$ its Borel
$\sigma $-field. By the 
Bochner-Minlos theorem \cite{Minlos}, there exists a unique
probability measure $\mathrm{d}\mathcal{P}$ defined on
$\mathcal{B}(\cS^{\prime}(\R))$ such that  
\begin{equation*}
\E\left[ e^{i\Re \langle \cW,\phi \rangle
	   }\right]
:=\int_{\cS^{\prime}(\R)}e^{i\Re \langle
w,\phi \rangle }\mathrm{d}  
\mathcal{P}(w)=e^{-\frac{1}{4}\Vert \phi \Vert _{L^{2}(\R
)}^{2}},\quad \phi \in \cS(\R). 
\end{equation*}
Here $\langle \cW,\phi \rangle =\cW(\phi)$ for $\cW\in 
\cS^{\prime}(\R)$ and $\phi \in \cS(\R)$, and 
$\E$ is the expectation with respect to $\mathrm{d}\mathcal{P}$. We
obtain this way a `stochastic Gelfand triple'
$(\cS^{\prime}(\R),\mathcal{B}(\cS^{\prime}(\R)),\mathrm{d}\mathcal{P})$,
\emph{the white noise probability space}, where $\mathrm{d}\mathcal{P}$ is
called the white noise measure. If $\cW$ is a random
variable with distribution $\mathcal{P}$ and given $\phi _{1},...,\phi
_{n}\in \cS(\R)$ orthonormal in $L^{2}(\R)$, $(\langle \cW,\phi _{1}\rangle ,...,\langle \cW,\phi
_{n}\rangle )$ is a complex standard Gaussian random
vector. Now, consider $\cW$ as a random variable
with distribution $\mathcal{P}$. For $\phi 
\in \cS(\R)$, the sequence $\{\langle \phi _{n},\phi \rangle
\}_{n=0}^{\infty }$ is square summable, and then $\langle \cW,\phi
\rangle =\sum_{n=0}^{\infty }\langle \cW,\phi _{n}\rangle \langle
\phi _{n},\phi \rangle $ $\mathcal{P}$-a.s. since $\langle \cW,\phi
_{n}\rangle \sim N_{\C}(0,1)\ (n=0,1,\dots )$ are
i.i.d. 
Therefore, $\mathcal{P}$-almost surely, $\cW$ can be expanded as
\begin{equation}
\cW=\sum_{n=0}^{\infty }\langle \cW,\phi _{n}\rangle \phi
_{n}\quad \text{in $\cS^{\prime}(\R)$,}  
\label{white}
\end{equation}
with $\langle \cW,\phi _{n}\rangle \sim N_{\C}(0,1)$ i.i.d.

White noise admits a natural interpretation in the
stochastic Gelfand triple described above. 
In the next subsection, 
we explain how our main results apply to random time-frequency analysis. 
Within this context, there are two main directions of research:

\begin{enumerate}
\item Recovery of a signal embedded in white noise from its spectrogram
zeros. This is a recent research topic, which leverages the properties of
the zeros of the GEF for signal-recovery methods 
\cite{PNAS,Silence0,BFC,Maxima,Escudero,BH,BardenetSampta,GWHF}.

\item High resolution time-frequency representations. Averages of
spectrograms with different windows improve the resolution of the
representation. They first appeared in time-varying power spectral
statistical estimation \cite{BB,FranzH} and, more recently, in high
resolution time-frequency representations \cite{XFlandrin,ConceFT,ImpConcEFT}. 
To the best of our knowledge, this is the first paper 
to investigate the efficiency of these methods using properties of zeros.
\end{enumerate}

\subsection{Time-frequency analysis and white noise spectrograms}\label{subsec:time-frequency}

Some basic concepts of time-frequency analysis are required to understand
the representation of white noise in the time-frequency plane. Given a
`window' function $g\in L^{2}({\R})$, the short-time Fourier
transform of the `signal' $f\in L^{2}({\R})$ is \cite{Charly}: 
\begin{equation*}
V_{g}f(x,\xi)=\int_{{\R}}f(t)\overline{g(t-x)}e^{-2\pi i\xi
t}dt,\qquad (x,\xi )\in {\R^{2}}. 
\label{eq_stft}
\end{equation*}
A fundamental property of the STFT is the Moyal formula: 
given $f_{1},f_{2},g_{1},g_{2}\in L^{2}(\R)$, 
\begin{equation*}
\left\langle
 V_{g_{1}}f_{1},V_{g_{2}}f_{2}\right\rangle_{L^{2}(\R^{2})}=\left\langle
 f_{1},f_{2}\right\rangle_{L^{2}(\R)}\left\langle 
g_{1},g_{2}\right\rangle_{L^{2}(\R)}.  
\label{Moyal_STFT}
\end{equation*}
This implies that, when $\lVert g\rVert _{2}=1$, the {Moyal formula} shows
that the map $V_{g}$ is an isometry between $L^{2}({\R})$ and a
closed subspace of $L^{2}({\R^{2}})$. Recall that Feichtinger's
algebra $\cS_{0}(\R)$ can be defined as 
\begin{equation*}
\cS_{0}(\R):=\left\{ g\in {L}^{2}(\R):V_{g}g\in
L^{1}(\C)\right\}. 
\end{equation*}
The Feichtinger algebra $S_0(\R)$ is a convenient
test-function space for the STFT \cite{F81}, and one may use the \emph{Gelfand triple}
$(\cS_{0}(\R),L^{2}\left( \R\right), \cS_{0}^{\prime }(\R))$ instead of 
$(\cS(\R),L^{2}\left( \R\right), \cS^{\prime }(\R))$ 
in the white noise construction of Section~\ref{sec:WN} and obtain in a
similar way a `stochastic Gelfand triple'
$(\cS_{0}^{\prime}(\R),\mathcal{B}(\cS_{0}^{\prime
}(\R)),\mathrm{d}\mathcal{P})$, which will be
\emph{the complex white noise probability space.} 
The STFT of $f\in \cS^{\prime }(\R)$ (or
$\cS_{0}^{\prime }(\R)$) with respect to a
window function $g\in \cS(\R)$ (or
$\cS_{0}(\R)$) is 
\begin{equation}
V_{g}f(x,\xi ):=\langle f,M_{\xi }T_{x}g\rangle ,\quad
 (x,\xi) \in \R^2. 
\label{distributionalSTFT}
\end{equation}
We will select $g\in \cS_{0}(\R)$ among Hermite functions
normalized as follows: 
\begin{equation*}
h_{n}(t)=\frac{2^{1/4}}{\sqrt{n!}}\left( \frac{-1}{2\sqrt{\pi }}\right)
^{n}e^{\pi t^{2}}\frac{d^{n}}{dt^{n}}\left( e^{-2\pi t^{2}}\right) ,\qquad
n\geq 0. 
\end{equation*}
We identify $(x,\xi )\in \R^{2}$ with $z\in
\C$ by $z:=x+i\xi $. Then, computing $V_{h_{n}}h_{k}$\ as in \cite{Abr2010} and
using (\ref{ComplexHermite}), we have 
\begin{equation}
V_{h_{n}}h_{k}\left( \frac{\bar{z}}{\sqrt{\pi }}\right)
 =e^{ix\xi }e^{-\frac{\left\vert z\right\vert
 ^{2}}{2}}H_{k,n}(z,\zbar).  
\label{hermiteT}
\end{equation}
Using (\ref{hermiteT}) and (\ref{FnHermite}), we can take
the STFT of (\ref{white}) in the distributional sense
(\ref{distributionalSTFT}) and obtain 
the following expansion, convergent as an $L^{2}(\mathcal{P})$-limit: 
\begin{equation*}
V_{h_{n}}\cW\left( \frac{\bar{z}}{\sqrt{\pi
		     }}\right) =e^{ix\xi
}e^{-\frac{\left\vert z\right\vert
^{2}}{2}}\sum_{k=0}^{\infty }\zeta 
_{k}H_{k,n}(z,\zbar)=e^{ix\xi }e^{-\frac{\left\vert z\right\vert^{2}
}{2}}f_{n}(z,\zbar), 
\label{calc}
\end{equation*}
where\ $\zeta _{k}=\langle \cW,h_{k}\rangle \sim N_{\C}(0,1)$
i.i.d.\ The spectrogram is defined as the square intensity function of $
V_{g}f(x,\xi )$: $\spec_{g}f(z)=\left\vert V_{g}f(x,\xi )\right\vert ^{2}$.
Thus, the spectrogram of white noise with Hermite windows is 
\begin{equation}
\spec_{h_{n}}\cW(\frac{\bar{z}}{\sqrt{\pi }})
=\left\vert f_{n}(z,\zbar)e^{-\frac{\left\vert z 
\right\vert^{2}}{2}}\right\vert^{2} 
\label{spec}
\end{equation}
and the zeros of $\spec_{h_{n}}\cW(z)$ coincide with the zeros of $
f_{n}(z,\zbar)$ 
under the identification $z = x+i\xi$ and the scaling in
\eqref{spec}. 
We thus can rewrite Theorem~\ref{Main} in spectrogram
language as follows.

\begin{theorem}\label{Mainprime1}
For fixed $k \ge 1$, as $n \to \infty$, 
the short-range correlations between the zeros of
 $\spec_{h_{n}}\cW(z)$ and
 $\spec_{h_{n+k}}\cW(z)$ are 
\begin{equation*}
g_{n,n+k}(z,z) = 
\begin{cases}
\disp 1-\frac{5}{4n^{2}}+O\left( \frac{1}{n^{3}}\right) <1 & \text{if $k=1$}, \\[4mm]
\disp \frac{4}{3}+
\frac{1}{12n^2}+O\left( \frac{1}{n^3}\right) 
>1 &
 \text{if $k=2$}, 
\\[4mm]
\disp 1 & \text{if $k\geq 3$}.  
\end{cases}
\end{equation*}
Thus, the short-range correlations
between zeros of $\spec_{h_{n}}\cW(z)$ and zeros of 
$\spec_{h_{n+k}}\cW(z)$ display repulsion if $k=1$, 
attraction if $k=2$, and no short-range second-order correlation in the contact
 limit if $k\geq 3$. Moreover, $g_{n,n+1}(z,z)$ increases
 with $n$ and $g_{n,n+1}(z,z) \to 1$, and
$g_{n,n+2}(z,z)$ attains its maximum at $n=2$ and then
 decreases for $n\geq 2$, 
with $g_{n,n+2}(z,z) \to 4/3$ as $n \to \infty$. 
\end{theorem}

\begin{figure}[tbp]
\centering
\includegraphics[scale=0.6]{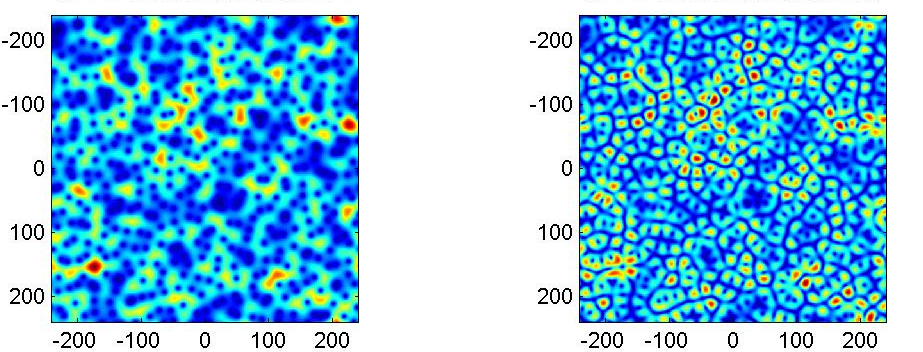}
\caption{The white noise spectrograms with the Gaussian (left) and the first
Hermite window (right). Dark blue areas correspond to low energy (around
zeros) and orange to high energy (around local maxima). One can observe that
the locations with local maxima of the first picture have low energy on the
second, while the local maxima of the second picture tend to occupy low
energy regions of the first.}
\end{figure}

Flandrin \cite[Section 10.2]{Flbook} has observed that the local maxima of
consecutive spectrograms tend to occur in different places, so that the
energy patches of the whole sequence $\{\spec_{h_{n}}\cW
(z)\}_{n=0}^{\infty }$ tend to be negatively correlated with each other.
The heuristics involving the allocation of energy patches is supported by
Theorem~\ref{Mainprime1}: zeros of consecutive $\spec_{h_{n}}\cW(z)$ tend
to occur in different places and, with exception of the case
$k=2$ in Theorem~\ref{Mainprime1}, the zeros of the remaining sequence
$\{\spec_{h_{n}}\cW(z)\}_{n=0}^{\infty }$ 
tend to be uncorrelated with each other in the contact limit. 
Such considerations provide a rationale for the use of multitaper combined
with reassignment
methods as a tool to sharpen spectrogram resolution \cite{XFlandrin,ConceFT}
and suggest that, as $N$ grows, the average 
\begin{equation}
\frac{1}{N}\sum_{n=0}^{N-1}\spec_{h_{n}}\cW(z)  \label{average}
\end{equation}
provides increasingly better estimates for the flat spectrum
of the white noise. In practice a reduced number of windows is used 
\cite{ConceFT,ImpConcEFT}; optimal values of $N=4$ have been observed in the
supplementary material of \cite{ConceFT}, and $N=2$ in the Morse wavelet
case, which has a similar structure, with the first transform applied to
analytic white noise \cite{BH,KolSapmTA} being a hyperbolic Gaussian
Analytic Function \cite{peres2005zeros,SHIRAIHyperbolic}. These curious
optimality properties are explained by analyticity of the first transform 
\cite{AscBrun,AnalyticWavelet}, the optimal concentration properties of the
first Hermite and Morse functions \cite{DauTF,DauTS}, and by the
correlations in Theorem~\ref{Mainprime1}. In general, the negative and low
correlations among energy patches of the sequence $\{\spec_{h_{n}}\cW
(z)\}_{n=0}^{\infty }$ support, in the case of white noise, the efficiency
of the multitaper Wigner-Ville \cite{BB,FranzH} spectral estimation methods
in \cite{BB}, using the average (\ref{BB_0}). In light of \cite[Section
10.2]{Flbook}, Theorem~\ref{Mainprime1} already suggests that the $N\rightarrow
\infty $ limit of the average (\ref{average}) for $\nu =\cW$ is $1$.
Using \eqref{spec}, Theorem~\ref{thm:functionalLLN} and
Theorem~\ref{thm:functional-clt} yield the following
consequence. 

\begin{theorem}\label{thm:functionalLLN-prime}
For every compact $K\subset \C$, as $N \to \infty$, 
\begin{equation*}
\frac{1}{N}\sum_{k=0}^{N-1} \spec_{h_{k}}\cW(z) 
\to 1 \quad \text{a.s. in $C(K)$}.
\end{equation*}
\end{theorem}
As mentioned before, this theorem can be strengthened to 
convergence in $C^{0,\alpha}(K)$ for any $0<\alpha<1$.

\begin{theorem}\label{thm:functionalCLT-prime}
For every compact $K\subset \C$, as $N \to \infty$, 
\begin{equation*}
\frac{1}{\sqrt{N}} \sum_{k=0}^{N-1}
\Big( \spec_{h_{k}}\cW(\sqrt{N} \cdot ) - 1 \Big) 
\inlawto \widetilde{\cG} \quad \text{ in $C(K)$},  
\end{equation*}
where $\widetilde{\cG}$ is the Gaussian process 
with covariance $\kappa(\sqrt{\pi}|\xi-\eta|)$. 
\end{theorem}

\subsection{Remarks on spectrogram averages}

We conclude with two remarks on spectrogram averages. First, we
compare the white-noise case with deterministic phase-space
localization. Second, we discuss a possible correlation-aware
choice of finitely many Hermite windows. 

To understand how Theorems~\ref{thm:functionalLLN-prime} and
\ref{thm:functionalCLT-prime} reflect the lack of concentration of white-noise spectrograms, one can compare
Theorem~\ref{thm:functionalLLN} with the deterministic results in
\cite{AGR, KLS25}, where it was shown that, for a deterministic rotationally invariant
window $g$, 
\begin{equation*}
\frac{1}{N}\sum_{k=0}^{N-1}\spec_{h_{k}}g(z)\rightarrow 
\chi_{\mathbb{D}}(z) \quad \text{in $L^1$},  
\end{equation*}
where $\mathbb{D}$ is the disk determined by the phase-space
localization of $g$. 
Thus, the phase space localization of the signal determines the
limit of spectrogram averages. In the deterministic case, 
the rotational invariance of $g$ leads to a
phase-space representation $V_{h_{k}}g(z)$ essentially supported on circular
domains, and this is reflected on the average asymptotics. 
By contrast, for white noise, the i.i.d. behavior of
$V_{h_{k}}\cW(z)$ implies that $\cW$ is
well-spread among the entire plane $\C$, leading
instead to the flat limit $\chi_{\C}(z)=1$.

Theorem~\ref{Mainprime1} also suggests that one may obtain a
better approximation to the spectrum
by averaging pairs of consecutive windows separated by gaps
of two levels, thereby avoiding positively
correlated pairs at distance $2$. More precisely, this excludes the pairs
involving $\spec_{h_{4k}}\cW(z)$ and
$\spec_{h_{4k\pm 2}}\cW(z)$, as well as those involving
$\spec_{h_{4k+1}}\cW(z)$ and
$\spec_{h_{4k+1\pm 2}}\cW(z)$. This motivates the alternative
average
\begin{equation*}
\frac{1}{2N}\sum_{k=0}^{N-1}
\left[\spec_{h_{4k}}\nu(z)+\spec_{h_{4k+1}}\nu(z)\right].
\end{equation*}
Heuristically, this averaging scheme may yield better results than
\eqref{BB_0}, both in multitapering \cite{BB} and in methods for sharpening
spectrogram resolution \cite{XFlandrin,ConceFT,ImpConcEFT}.

\begin{acknowledgement}
We are grateful to Zouha\"ir Mouayn for inviting us to the
conference ICMMP19 in Marrakech in 2019, where this project was initiated.
We also thank R\'emi Bardenet, who suggested that the experiments reported
in [34, Section~10.2] seemed to point to the existence of a kind of
``planar interlacing property.'' We are glad to confirm his intuition:
at least for zeros, such a pattern exists. We thank Patrick Flandrin and
Hau-Tieng Wu for encouraging comments on earlier versions of
 this manuscript.
L.~D.~Abreu was supported by FWF Project
 10.55776/PAT8205923. T.~Shirai was supported by JSPS
 KAKENHI Grant Numbers JP18H01124, JP20K20884, JP22H05105,
 JP23H01077, JP23K25774 and also supported in part by
 JP20H00119 and JP21H04432.
\end{acknowledgement}

\end{document}